\documentclass{amsart}

\usepackage[hmargin=3.5cm,top=3cm,bottom=3cm,includehead]{geometry}

\usepackage{amsmath}
\usepackage{amssymb}
\usepackage{amsthm}

\usepackage{enumerate,color}
\usepackage{hyperref}

\def\pa{\partial}
\def\var{\varepsilon}
\def\R{{\Bbb R}}
\def\Z{{\Bbb Z}}
\def\la{\langle}
\def\ra{\rangle}

\newcommand{\be}{\begin{equation}}
\newcommand{\ee}{\end{equation}}
\newcommand{\ba}{\begin{aligned}}
\newcommand{\ea}{\end{aligned}}

\newtheorem{Defi}{Definition}
\newtheorem{Thm}{Theorem}
\newtheorem{Prop}[Thm]{Proposition}
\newtheorem{Cor}{Corollary}[Thm]
\newtheorem{lem}[Thm]{Lemma}
\theoremstyle{remark}
\newtheorem{rem}{Remark}[Thm]

\newcommand{\ben}{\begin{eqnarray}}
\newcommand{\een}{\end{eqnarray}}
\newcommand{\beno}{\begin{eqnarray*}}
\newcommand{\eeno}{\end{eqnarray*}}

\begin{document}

\title[Large time behavior of Landau equation]{Estimates for the large time behavior\\ of the Landau equation\\ in the Coulomb case}

\author{K.~Carrapatoso}
\author{L.~Desvillettes}
\author{L.~He}

\address[Kleber Carrapatoso]{Centre de Math\'ematiques et de Leurs Applications, ENS Cachan, CNRS, Universit\'e Paris-Saclay, F-94235 Cachan, France.}
\email{carrapatoso@cmla.ens-cachan.fr}

\address[Laurent Desvillettes]{Univ. Paris Diderot, Sorbonne Paris Cit\'e, Institut de Math\'ematiques de Jussieu - Paris Rive Gauche, UMR 7586, CNRS, Sorbonne Universit\'es, UPMC Univ. Paris 06, F-75013, Paris, France.} 
\email{desvillettes@math.univ-paris-diderot.fr}

\address[Lingbing He]{Department of Mathematical Sciences, Tsinghua University, Beijng, 100084, P.R. China.}
\email{lbhe@math.tsinghua.edu.cn}

\date{\today}

\begin{abstract}
This work deals with the large time behaviour of the spatially homogeneous Landau equation with Coulomb potential. 
Firstly, we obtain a bound from below of the entropy dissipation $D(f)$ by a weighted relative Fisher information of $f$ with respect to the associated Maxwellian distribution, which leads to a variant of Cercignani's conjecture thanks to 
a logarithmic Sobolev inequality. Secondly, we prove the propagation of  polynomial and stretched exponential moments with an
at most linearly growing in time rate. As an application of these estimates, we show the convergence of any ($H$- or weak) solution to the Landau equation with Coulomb potential to the associated Maxwellian equilibrium with an explicitly computable rate, assuming initial data with finite mass, energy, entropy and some higher $L^1$-moment. More precisely, if the initial data have some (large enough) polynomial $L^1$-moment, then we obtain an algebraic decay. If the initial data have a stretched exponential $L^1$-moment, then we recover a stretched exponential decay. 
\end{abstract}



\maketitle


\section{Introduction} \label{intro}
The Landau equation is a fundamental model in kinetic theory that describes the evolution in time of a plasma due to collisions between charged particles.
\par
We consider in this work the spatially homogeneous Landau equation with Coulomb potential (cf. \cite{lipi,CC,DV1})
\be\label{landau}
\partial_t f = Q(f,f),
\ee
which is complemented with initial data $f_0=f_0(v) \ge 0$. Here $f := f(t,v) \ge 0$ stands for the distribution of particles that at time $t \in \R_+$ possess the velocity $v \in \R^3$. The Landau operator $Q$ is a bilinear operator acting only on the velocity variable $v$. It writes
\be\label{Q}
Q(f,f)(v) = \nabla \cdot \int_{\R^3} a(v-w) \{ f(w) \nabla f(v) - \nabla f(w) f(v)     \} \, dw,
\ee
where $a$ is a matrix-valued function that is symmetric, (semi-definite) positive.
It depends on the interaction potential between particles, and is defined by (for $i,j =1, 2, 3$)
\be\label{a}
a_{ij}(z) = |z|^{\gamma+2}\, \Pi_{ij} (z), \quad \Pi_{ij}(z)= \delta_{ij} - \frac{z_i z_j}{|z|^2}, \quad  -4 < \gamma \le 1.
\ee
Observe that $\Pi(z) := (\Pi_{ij}(z))_{i,j = 1,2,3}$ is the orthogonal projection onto $z^\perp$. One usually classifies the different cases as follows: hard potentials $0<\gamma \le 1$, Maxwellian molecules $\gamma = 0$, moderately soft potentials $-2 \le \gamma < 0$ and very soft potentials $-4 < \gamma <-2$. Note that the very soft potentials include the Coulomb potential $\gamma=-3$. 
From now on, for the sake of clarity and because it is the most physically interesting case, we shall only consider in this work the Coulomb potential case $\gamma=-3$, except when moment estimates are concerned. 
It is however worth mentioning that our methods and results can be straightforwardly adapted to the soft potentials case $-4 < \gamma < 0$, as explained in more details in the remarks after our main theorems.

One usually introduces the quantities
$$
b_i(z) = \sum_{j=1}^3 \partial_j a_{ij}(z) = -2 \,z_i\, |z|^{-3}, \quad 
c(z) = \sum_{i=1}^3\sum_{j=1}^3\partial_{ij} a_{ij}(z) = -8 \pi\, \delta_0(z),
$$
so that the Landau operator can be written as
\be\label{Qbis}
\ba
Q(f,f) 
&= \sum_{i=1}^3 \partial_i \bigg( \sum_{j=1}^3(a_{ij}*f)\, \partial_j f - (b_i *f)\, f  \bigg)   \\
&= \sum_{i=1}^3\sum_{j=1}^3(a_{ij}*f) \,\partial_{ij}f + 8\pi \,f^2 .
\ea
\ee

\medskip

At the formal level, we can write thanks to \eqref{Q} a weak formulation of the Landau operator $Q$, for a test function $\varphi$, 
in the following way:
\be\label{Qweak1}
\ba
&\int_{\R^3} Q(f,f) (v) \varphi(v) \, dv \\
&\qquad
= - \frac12 \, \sum_{i=1}^3\sum_{j=1}^3 \iint_{\R^3 \times \R^3} a_{ij}(v-w) \left\{ \frac{\partial_i f}{f}(v) - \frac{\partial_i f}{f}(w)  \right\} \left\{ \partial_j \varphi(v) - \partial_j \varphi(w)  \right\}   f(v) f(w) \, dv \, dw .
\ea
\ee
Another weak formulation, based on \eqref{Qbis}, also holds at the formal level:
\be\label{Qweak2}
\ba
\int_{\R^3} Q(f,f) (v) \varphi(v) \, dv 
&= \frac12 \sum_{i=1}^3\sum_{j=1}^3 \iint_{\R^3 \times \R^3} a_{ij}(v-w) \big\{ \partial_{ij}\varphi(v) + \partial_{ij} \varphi(w) \big\} \, f(v) f(w) \, dv \, dw \\
&\quad
+ \sum_{i=1}^3 \iint_{\R^3 \times \R^3} b_{i}(v-w) \big\{ \partial_{i}\varphi(v) - \partial_{i} \varphi(w) \big\} \, f(v) f(w) \, dv \, dw.
\ea
\ee
From the weak formulation (\ref{Qweak1}), we can easily deduce some fundamental properties of the Landau operator $Q$. The operator indeed conserves (at the formal level) mass, momentum and energy, more precisely
\be\label{cons}
\int_{\R^3} Q(f,f)(v) \varphi(v) \, dv = 0 \quad\text{for}\quad \varphi(v) = 1, v_i, \frac{|v|^2}2.
\ee
We also deduce from (\ref{Qweak1}), at the formal level, the entropy 
structure
by taking the test function $\varphi(v) = \log f(v)$, 
that is 
\be\label{D(f)}
D(f)  := - \int Q(f,f)(v) \log f(v) \, dv
\ee
$$ =  \frac12 \sum_{i=1}^3\sum_{j=1}^3\iint_{\R^3 \times \R^3} a_{ij}(v-w) \left\{ \frac{\partial_i f}{f}(v) - \frac{\partial_i f}{f}(w)  \right\} \left\{ \frac{\partial_j f}{f}(v) - \frac{\partial_j f}{f}(w)  \right\}   \, f(v)\, f(w) \, dv \, dw.  $$
Note that $D(f) \ge 0$ since the matrix $a$ is (semi-definite) positive. 
It also follows (see for example \cite{DV2}) that any equilibrium (that is, any $f$ such that $D(f)=0$) is a Maxwellian distribution
\be\label{MaxGen}
\mu_{\rho,u,T}(v) = \frac{\rho}{(2\pi T)^{3/2}} \, e^{-\frac{|v-u|^2}{2T}},
\ee
where $\rho\ge 0$ is the density, $u \in \R^3$ the mean velocity and $T>0$ the temperature, defined by
$$
\rho = \int_{\R^3} f(v) \, dv, \quad 
u = \frac{1}{\rho} \int_{\R^3} v \,f(v) \, dv, \quad
T = \frac{1}{3 \rho} \int_{\R^3} |v-u|^2\, f(v) \, dv.
$$
As a consequence of the properties above at the level of the operator, 
we can obtain the corresponding properties (at the formal level) for
the solutions of the spatially homogeneous Landau equation (\ref{landau}),
that is, the conservation of mass, momentum and energy
\be\label{conssh}
\forall t \ge 0, \qquad \rho(t)  = \rho (0), \quad u(t)=u(0),
\quad T(t)=T(0),
\ee
on one hand, and the entropy property on the other hand
\be\label{entr}
\frac{d}{dt}H(f(t,\cdot)) = -D(f(t,\cdot)) \le 0,
\ee
where $H(f) := \int f(v) \log f(v) \, dv$ is the entropy and 
$D(f)$, defined by (\ref{D(f)}), is the entropy dissipation.
\medskip

Throughout this paper, we shall always assume that $f_0 \ge 0$ and $f_0 \in L^1_2 \cap L \log L (\R^3)$. Furthermore, in most of the paper, we suppose, without loss of generality, that $f_0$ satisfies the normalization identities 
\be\label{f0}
\int_{\R^3} f_0(v)\, dv = 1, \quad \int_{\R^3} f_0(v)\, v\, dv =0, 
\quad \int_{\R^3} f_0(v)\, |v|^2 \, dv =3,
\ee
which can be rewritten $\rho(0)=1$, $u(0)=0$, $T(0)=1$.
Finally, we denote by $\mu(v) = (2\pi)^{-3/2} e^{-|v|^2/2}$ the Maxwellian distribution (centred reduced Gaussian) with same mass, momentum and energy as $f_0$ satisfying (\ref{f0}).
\medskip

Let us briefly recall some existing results on the Landau equation \eqref{landau} with Coulomb potential. Villani~\cite{Vi} proved global existence of the so-called $H$-solutions for initial data with finite mass, energy and entropy. More recently, the second author \cite{D} proved that $H$-solutions are in fact weak solutions (in the usual sense), thanks to a new estimate for the entropy dissipation $D(f)$. More precisely it is obtained in \cite{D} that there is an explicitly computable constant $C_0 = C_0(\bar{H}) >0$ such that, for all (normalized) $f\ge 0$ satisfying $H(f) \le \bar{H}$, the following inequality holds:
\be\label{L3<D}
\| f \|_{L^3_{-3}} \le C_0 \,( 1 + D(f) ),
\ee
where (for any $p \in [1, +\infty[$, $q \in \R$) the $L^p_{q}$ norm is defined by
$$ \| f \|_{L^p_{q}}^p = \int_{\R^3} |f(v)|^p \, (1 + |v|^2)^{pq/2}\, dv . $$
Therefore, since (for $H$ solutions of the spatially homogeneous Landau equation)  $D(f) \in L^1_t (]0,\infty[)$ thanks to identity (\ref{entr}), we obtain that any $H$ solution of this equation lies in $ L^1_{loc}([0,\infty); L^3_{-3}(\R^3))$, which is sufficient to define weak solutions in the usual sense (using the weak form \eqref{Qweak2}), see \cite{D} for more details. We also quote \cite{Vi2} for 
renormalized solutions in the spatially inhomogeneous context, and \cite{Alli} for local in time solutions.
\medskip

Let us mention the results concerning the well-posedness issue. Fournier \cite{Fournier} obtained that uniqueness holds in the class $L^\infty_{loc}([0,\infty); L^1_2(\R^3)) \cap L^1_{loc}([0,\infty);L^\infty(\R^3))$, and this result implies a local well-posedness result assuming further that the initial data lie in $L^\infty (\R^3)$, thanks to the local existence result of Arsenev-Peskov~\cite{Arsenev} for such initial data. We also refer to \cite{Guo} and  \cite{heyang} for the global well-posedness and the local well-posedness for the inhomogeneous equation in weighted Sobolev spaces, as well as to \cite{AMUXY1,AMUXY2,GreStr} for the non-cutoff Boltzmann equation, whose structure shares similarities with the Landau equation.

\medskip

Concerning the large time behaviour issue, we shall mention some known results for all kind of potentials. In the spatially homogeneous case, Villani and the second author~\cite{DV2} proved exponential decay to equilibrium in the Maxwellian molecules case $\gamma=0$, and algebraic decay for hard potentials $0 < \gamma \le 1$. Later, the first author~\cite{KC1} proved exponential decay for hard potentials. Toscani and Villani~\cite{TosVi-slow} proved algebraic decay for mollified soft potentials $-3 < \gamma < 0$ (i.e.\ truncating the singularity of \eqref{a} at the origin) excluding the Coulomb case, and the first author~\cite{KC2} proved polynomial convergence for moderately soft potentials $-2 < \gamma < 0$ and exponential convergence in the case $-1 < \gamma < 0$. 
Some results were also obtained in the spatially inhomogeneous case. 
For potentials in the range $-2 \le \gamma \le 1$ and in a close-to-equilibrium framework, exponential decay to equilibrium has been established by Mouhot and Neumann~\cite{MouNeu}, 
Yu~\cite{Yu}, and more recently by Tristani, Wu and the first author~\cite{CTW}. 
Still in a perturbative framework and for the Coulomb case, Guo and Strain~\cite{GS} (see also \cite{GS2}) proved stretched exponential decay to equilibrium in a high-order Sobolev space with fast decay in the velocity variable. 
Also, for general initial data and in the Coulomb case, Villani and the second author~\cite{DV-boltzmann} proved algebraic convergence to equilibrium for (uniformly w.r.t time) {\it{a priori}} smooth solutions.

\medskip

The aim of this work is to study the large time behaviour of solutions to the spatially homogeneous Landau
equation in the Coulomb case. Our proof is based on an entropy-entropy dissipation method.
\par
This method (and its variants) has been widely used to tackle the large time behaviour of many models in kinetic theory (cf. in particular \cite{TO, DV2, TV1}, and earlier attempts like \cite{D2})
as well as in many other PDEs or integral equations (cf. for example \cite{AB} or \cite{DF1}).  It is important to emphasize that this method can handle nonlinear equations directly (that is, no linearization is involved).
\par
 Roughly speaking, it consists in looking for some Lyapunov functional for the evolution equation (usually called entropy) and then in computing its associated dissipation (usually called entropy dissipation). Then, the existence of  functional inequalities relating the entropy dissipation to the entropy itself is investigated. 
When the method is successful, such 
inequalities enable to close a differential inequality for the entropy, and yield the large time behaviour.
\par
When the functional inequality involves quantities which grow slowly (that is, polynomially) with respect to time along the flow
of the equation, the entropy-entropy dissipation method is said to be ``with slowly growing {\it{a priori}} bounds''. We refer
for example to \cite{TosVi-slow} and \cite{DF2} for such a situation. In this work, we also use this variant of
the entropy-entropy dissipation method.
\par
In kinetic theory, more precisely when Boltzmann and Landau equations are concerned, the functional inequality that hopefully links the entropy dissipation and the entropy was suggested by Cercignani (cf. \cite{cerc}), and has been known since as Cercignani's conjecture. We refer
to \cite{DMV} for a detailed description of the network of conjectures now bearing this name. We present in this work a variant
of the so-called weak Cercignani's conjecture for the Landau equation (with Coulomb potential).

\section{Main results}\label{sec:main}
We state in this section our main results.
Hereafter we shall denote polynomial $L^1$-moments (for $\ell\in \R$) by
\be\label{moment-poly}
M_\ell (f) := \| f \|_{L^1_\ell(\R^3)} := \int_{\R^3} \la v \ra^\ell \, f(v) \, dv, \quad \la v \ra := (1 + |v|^2)^{1/2},
\ee
as well as stretched exponential $L^1$-moments (for $s>0$, $\kappa\in \R$) by
\be\label{moment-exp}
M_{s,\kappa} (f) := \| f \|_{L^1(e^{\kappa \la v \ra^s}\,dv)} := \int_{\R^3} e^{\kappa \la v \ra^s}\, f(v) \, dv.
\ee

\medskip

Our first main result is a new estimate that bounds from below the entropy dissipation $D(f)$ (defined in \eqref{D(f)}) by a weighted relative Fisher information of $f$ with respect to the associated Maxwellian distribution $\mu$, provided that 
some higher moment of $f$ is controlled. 

\begin{Thm}\label{thm:entropy}
One can find $C:= C(\bar{H})>0$ depending only on $\bar{H}$
such that for all $f\ge 0$ satisfying (the normalization of mass, momentum and energy)
\begin{equation} \label{un}
\int_{\R^3} f(v)\, dv = 1, \quad \int_{\R^3} f(v)\, v\, dv =0, 
\quad \int_{\R^3} f(v)\, |v|^2 \, dv =3,
\end{equation}
and also satisfying (an upper bound on the entropy)
\begin{equation} \label{unbis}
H(f) := \int_{\R^3} f(v)\, \log f(v) \, dv \le \bar{H},
\end{equation}
the following inequality holds:
\begin{equation} \label{infini}  
D(f) \ge C(\bar{H}) \, (M_5(f))^{-1}\, \int_{\R^3} f(v)\,  \bigg| \frac{\nabla f(v)}{f(v)} + v \bigg|^2 \, \la v \ra^{-3} \,dv .
\end{equation}
\end{Thm}

\begin{rem}\label{rem:thm:entropy}
\begin{enumerate}

\item  We consider in this work only the case of Coulomb potential, namely $\gamma = - 3$ in the definition of the matrix $a$ given by \eqref{a}. A straightforward adaptation also gives analogous results for general soft potentials $-4 < \gamma < 0$. 
In this situation, estimate (\ref{infini}) becomes
\begin{equation} \label{infinipr}   
D(f) \ge C(\bar{H}) \, (M_{2 - \gamma}(f))^{-1}\, \int_{\R^3} f(v)\,  \bigg| \frac{\nabla f(v)}{f(v)} + v \bigg|^2 \, \la v \ra^{\gamma} \,dv .
\end{equation}
 
 
\item

We recall that in the case of Maxwell molecules, that is, $\gamma =0$, estimate (\ref{infinipr}) is already known (cf. \cite{DV2}), and does not involve any higher moment of $f$ (it involves only $M_2(f)$).

\end{enumerate}

\end{rem}

The proof of this theorem is inspired by the arguments developed by the second author in \cite{D}, where it is obtained that the weighted (non relative) Fisher information of $f$ can be bounded from above by the entropy dissipation $D(f)$ plus some constant (depending on the mass and energy of $f$, when $f$ is not normalized). There are nevertheless important differences between the computations of \cite{D} and the proof given here. First, since we allow here the presence in the estimate of a moment of high order, one can use simpler multiplicators than in \cite{D} (no Maxwellian with an arbitrary temperature is introduced in the proof, cf. also \cite{desv_dspde3}). Secondly, and most importantly, one has to keep the exact value of the coefficients appearing in front 
of linear terms like $v_i$, whereas those terms were estimated without too much care in \cite{D}.
\medskip

As a consequence of estimate (\ref{infini}), we shall prove a variant of the so-called weak Cercignani's conjecture for the Landau equation (with Coulomb potential). We refer to \cite{DMV} for a systematic description of Cercignani's conjecture. Let us say here that the term ``weak'' means that some quantity other than the mass, energy and (upper bound on the) entropy plays a role 
in the relationship between $D(f)$ and a weighted version of the relative entropy. Indeed, we need here a control on the fifth moment of
$f$ (that is, $M_5(f)$). This result (variant of the weak Cercignani's conjecture) for the Landau equation (with Coulomb potential) is summarized
in the corollary below:

\begin{Cor}\label{cerc}
One can find $C:= C(\bar{H})>0$ depending only on $\bar{H}$
such that for all $f\ge 0$ satisfying (the normalization of mass, momentum and energy) (\ref{un}) and also satisfying (an upper bound on the entropy) (\ref{unbis}), 
the following inequality holds: 
\begin{equation} \label{cc}  
\ba
D(f) \ge C(\bar{H}) \, (M_5(f))^{-1} \,  \int  \left\{   f
\log \left( \frac{Z_1}{Z_2} \frac{f}{\mu}  \right) + \frac{Z_2}{Z_1} \mu - f \right\} \la v \ra^{-3} \, dv ,
\ea
\end{equation}
with $Z_1 = \int \la v \ra^{-3} \mu$ and $Z_2 = \int \la v \ra^{-3} f$. As a consequence, for any $R>0$ (and some absolute constant $C>0$)
\be\label{eq:D<H}   
\ba
D(f) 
&\ge C(\bar{H}) \, (M_5(f))^{-1} \, R^{-3} 
\Bigg( \int  f \log (f / \mu) \, dv 
- \int_{\la v \ra \ge R}  f \log f  \, dv \\
&\qquad\qquad\qquad
-C  \int_{\la v \ra \ge R} \la v \ra^2 \, f \, dv 
- C \int_{\la v \ra \ge R} \mu \, dv \Bigg).
\ea
\ee

\end{Cor}

\begin{rem}
As already explained in Remark \ref{rem:thm:entropy}-(1), this result can be easily adapted to the case of general soft potentials $-4 < \gamma < 0$, in which case we obtain estimates \eqref{cc} and \eqref{eq:D<H} replacing { $M_{5}(f)$ by $M_{2-\gamma}(f)$}, $\la v \ra^{-3}$ by $\la v \ra^\gamma$, $R^{-3}$ by $R^\gamma$, $Z_1$ by $\int \la v \ra^\gamma \mu$, and $Z_2$ by $\int \la v \ra^\gamma f$.
\end{rem}

\bigskip

As an application of the entropy dissipation estimates established in Theorem~\ref{thm:entropy} and Corollary~\ref{cerc}, we obtain the convergence (with rate) of any ($H$- or weak) solution $f$ (of the spatially homogeneous Landau equation with
Coulomb potential, and normalized initial data) to the associated Maxwellian equilibrium $\mu$, assuming only that the initial data has finite mass, energy, entropy and some higher $L^1$-moment. Before stating our result, let us introduce the notion of solutions that we shall consider in this work.


\begin{Defi}[$H$-solutions \cite{Vi}]\label{def:H-sol}
Consider a nonnegative $f_0 \in L^1_2 \cap L \log L (\R^3)$. We say that $f$ is a $H$-solution to the spatially homogeneous Landau equation~\eqref{landau} with Coulomb potential and with initial data $f_0$ if it satisfies:

\begin{enumerate}[\quad(a)]

\item $f \ge 0$, $f \in C ([0,\infty) ; \mathcal D' (\R^3)) \cap L^\infty ( [0,\infty) ; L^1_2 \cap L \log L (\R^3))$, $f(0) = f_0$;

\item The conservation of mass, momentum and energy, that is, for all $t \ge 0$,
$$
\int f(t,v) \, \phi(v) \, dv = \int f_0(v) \, \phi(v) \, dv \quad\text{for}\quad
\phi(v) = 1, v_j , |v|^2;
$$

\item The entropy inequality, for all $t \ge 0$,
$$
H(f(t)) + \int_0^t D(f(\tau)) \, d\tau \le H(f_0);
$$

\item $f$ satisfies \eqref{landau} in the distributional sense: for any test function $\varphi \in C^1 ([0,\infty) ; C^\infty_c(\R^3))$ and for any $t \ge 0$,
$$
\int f(t) \varphi \, dv - \int f_0 \varphi(0) \, dv - \int_0^t \int f(\tau) \partial_t \varphi (\tau) \, dv \, d\tau  = \int_{0}^t \int Q(f,f)(\tau) \, \varphi(\tau) \, dv \, d\tau ;
$$
where $\int Q(f,f) \, \varphi \, dv$ is defined by \eqref{Qweak1}.
\end{enumerate}

\end{Defi}

\begin{Defi}[Weak solutions]\label{def:w-sol}
Consider a nonnegative $f_0 \in L^1_2 \cap L \log L (\R^3)$. We say that $f$ is a weak solution to the spatially homogeneous Landau equation~\eqref{landau} with Coulomb potential and with initial data $f_0$ if it satisfies $(a)$, $(b)$, $(c)$, and $(d)$ of Definition~\ref{def:H-sol}, with the weak formulation of $\int Q(f,f) \, \varphi \, dv$ being defined by \eqref{Qweak2}.

\end{Defi}

As already mentioned, it was proven in \cite{D} that if $f \in L^\infty ([0,\infty) ; L^1_2 \cap L \log L (\R^3))$ and $D(f) \in L^1 ([0,\infty))$ then $f \in L^1_{loc} ([0,\infty) ; L^3_{-3} (\R^3))$, more precisely estimate \eqref{L3<D} holds. Therefore we can replace condition $(a)$ by 
$$
(a') \;  f \ge 0, \, f \in C ([0,\infty) ; \mathcal D' (\R^3)) \cap L^\infty ( [0,\infty) ; L^1_2 \cap L \log L (\R^3))  \cap L^1_{loc} ([0,\infty) ; L^3_{-3} (\R^3)), \, f(0) = f_0,
$$
and then the two notions of solutions are equivalent (because with this new bound we can define $\int Q(f,f) \, \varphi \, dv$ by \eqref{Qweak2}).

\medskip

Hereafter, in this work, we shall simply say that $f$ is a $H$- or weak solution to the Cauchy problem~\eqref{landau}, meaning that $f$ satisfies $(a')$ (with estimate \eqref{L3<D}), $(b)$, $(c)$ and $(d)$, with $\int Q(f,f) \, \varphi \, dv$ being defined equivalently by \eqref{Qweak1} or \eqref{Qweak2}. Moreover, we will sometimes split the operator $Q=Q_1 + Q_2$ and use $\int Q_1(f,f) \, \varphi \, dv$ defined by \eqref{Qweak1} and $\int Q_2(f,f) \, \varphi \, dv$ defined by \eqref{Qweak2}.

It is noticed in \cite{Vi} that \eqref{Qweak1} and \eqref{Qweak2} make sense as soon as $f$ satisfies $(a')$ and $\varphi=\varphi(v) \in W^{2,\infty} (\R^3)$.


\medskip

We can now state our second main result.

\begin{Thm}\label{thm:decay}
Let $f_0 \in L^1_2 \cap L \log L (\R^3) $ satisfy the normalization \eqref{f0}, 
and consider any global $H$- or weak solution $f$ to the spatially homogeneous Landau equation~\eqref{landau} with Coulomb potential and with initial data $f_0$.

\medskip

\begin{enumerate}[(i)]

\item Assume moreover that $f_0 \in L^1_\ell(\R^3)$ with  $ \ell > \frac{19}{2}$. 
Then for any positive $ \beta < \frac{2\ell^2-25\ell+57}{9(\ell - 2)} $, there exists some computable constant $C_{\beta}>0$ depending only on $\beta$, the initial entropy $H(f_0)$ and the initial moment $M_\ell(f_0)$, such that   
$$
 \forall\, t \ge 0, \qquad H(f(t) | \mu) \le C_\beta \,(1+t)^{ - \beta} .
$$

\medskip

\item Assume moreover that $f_0 \in L^1(\R^3, \,e^{\kappa \la v \ra^s}\, dv)$, with $\kappa >0$ and $s\in]0,1/2[$, or with
 $\kappa \in ]0,2/e[$ and $s=1/2$. Then there exist some computable constants $C,c>0$ depending only on $\kappa$, $s$, the initial entropy $H(f_0)$ and the initial moment $M_{s, \kappa}(f_0)$, such that
$$
 \forall\, t \ge 0, \qquad H(f(t) | \mu) \le C \, e^{-c (1+t)^{\frac{s}{3+s}}\, 
 (\log(1+t))^{ - \frac{3}{3+s}}} .
$$
\end{enumerate}

\end{Thm}

\begin{rem}\label{rem:thm:decay}
\begin{enumerate}

\item 
The normalization assumption \eqref{f0} is only for simplicity. The theorem also holds when the initial data are not normalized (i.e. for any $f_0 \in L^1_2 \cap L \log L (\R^3) $), up to the 
dependence of the constants and to a change in the limiting Maxwellian equilibrium to $\mu_{\rho, u, T}$ defined in \eqref{MaxGen}.

\item In point $(ii)$, the best rate of convergence towards equilibrium that we can achieve
 is in the case $s=1/2$, where we get a decay with a rate $ O (e^{- (1+t)^{\frac{1}{7}} \, (\log (1+t))^{ - \frac{6}{7}} })$.  We mention that in the close-to-equilibrium regime, the best decay rate is $O(e^{-t^{\frac23}})$, as 
 can be seen in \cite{GS}.

\item The restriction on the exponent $s \in ]0,1/2]$ comes from the results available on the propagation of stretched exponential moments (see Corollary~\ref{cor:moment-exp-VS}).

\item The estimates which are presented in the theorem above concern the relative entropy of the solution of the Landau
equation. Thanks to the Cziszar-Kullback-Pinsker inequality (cf. \cite{CKP1, CKP2}), they can be transformed in estimates on the $L^1$
norm of $f(t) - \mu$. Then, by interpolation, they also yield estimates for weighted $L^1$ norms of $f(t) - \mu$.

\item As in the case of Theorem \ref{thm:entropy} (see Remark~\ref{rem:thm:entropy}), it is possible to extend the estimates of Theorem \ref{thm:decay} 
to the Landau equation with general soft potentials $-4 < \gamma < 0$. The rates are then modified.

\end{enumerate}

\end{rem}

The proof of Theorem \ref{thm:decay} uses the entropy dissipation estimate of 
Theorem~\ref{thm:entropy} (more precisely, that of Corollary~\ref{cerc}) 
together with some interpolation inequalities, the regularity estimate~\eqref{L3<D} and the propagation of $L^1$-moments in the Coulomb case (see Lemma~\ref{lem:moment-poly-VS} and Corollary~\ref{cor:moment-exp-VS}).

It is worth mentioning that we do not follow the usual arguments in order to prove 
Theorem~\ref{thm:decay} (see e.g.\ \cite{TosVi-slow, KC2}). Indeed, after obtaining a weak form of Cercignani's conjecture as in \eqref{cc}, one usually obtains, thanks to some interpolation arguments, an inequality of the form
\be\label{dHdt-theta}
\frac{d}{dt} H(f | \mu) = - D(f) \le - K_{\theta}(f) \, H(f | \mu)^{1 + \theta}, \quad \theta>0,
\ee
where $K_\theta(f)$ is some functional depending on moments and some (high-order) regularity bounds on $f$. Then, in order to close the above differential inequality and conclude thanks to Gronwall's inequality, one needs to prove {\sl{a priori}} estimates for solutions $f$ (so that $K_\theta(f)$ can be controlled). However, when one considers the Coulomb potential, no  {\sl{a priori}} estimate is known for the high-order regularity of the solutions, the only 
regularity estimate at hand is indeed \eqref{L3<D}, that uses again the entropy dissipation $D(f)$. Thus, instead of using an inequality like \eqref{dHdt-theta}, we shall write a similar inequality, but keeping the exponent $1$ instead of $1+\theta$, at the price of some remainder term. We shall then use \eqref{L3<D} in order to control part of this remainder term, and only at the very end  shall we choose some interpolation (depending on time) in order to close a differential inequality and conclude thanks to some variant of Gronwall's lemma.

\section{Entropy dissipation estimate}
This section is devoted to the proof of Theorem~\ref{thm:entropy}.
Recall that we have defined in (\ref{D(f)}), for any $f:=f(v)\ge 0$, by 
$$
D(f)  := \frac12\, \sum_{i=1}^3 \sum_{j=1}^3 \iint_{{\Bbb R}^3\times {\Bbb R}^3} 
f(v)\,f(w)\,
|v-w|^{-1} \, \Pi_{ij}(v-w)\, \left ( \frac{\partial_i f}{f}(v) - \frac{\partial_i f}{f}(w) \right ) 
$$
$$ \times\,  \left( \frac{\partial_j f}{f}(v) - \frac{\partial_j f}{f} (w) \right )\, dv dw, 
$$
the entropy production of the Landau operator with Coulomb interaction, where $\Pi_{ij}(z)$, defined by (\ref{a}) is the $i,j$-component of the orthogonal projection $\Pi$ onto $z^\bot := \{ y\, / \; y\cdot z =0 \}$.
We also recall the notation 
$M_p(f)$
for the moment of $f$ of order $p$, and, for all $i,j \in \{1,2,3\}$, we define by
$$
P_{ij}^f = \int_{\R^3}  f(v)\, v_i \,v_j\, dv 
$$
the pressure tensor of $f$ (when $\int_{\R^3} f(v) \, v\, dv =0$). 

\bigskip

%

The proof of Theorem~\ref{thm:entropy} is a consequence of the three following Propositions \ref{prop:Rij}, \ref{prop:Delta} and \ref{prop:Deltabis}.

\begin{Prop}\label{prop:Rij}
We denote by $R_{ij}^f(v,w)$, for all $i,j \in \{1,2,3\}$, the quantity ($i,j$-component of the cross product of
$v-w$ and $\frac{\nabla f(v)}{f(v)} -  \frac{\nabla f(w)}{f(w)}$) 
\begin{equation}\label{zero}
R_{ij}^f(v,w) = (v_i - w_i)\, \left(\frac{\pa_j f(v)}{f(v)} -  \frac{\pa_j f(w)}{f(w)} \right)
 - (v_j - w_j)\, \left(\frac{\pa_i f(v)}{f(v)} -  \frac{\pa_i f(w)}{f(w)} \right). 
\end{equation}

Then, for all $f:=f(v)\ge 0$ such that (\ref{un}) is satisfied, and for all $i,j \in \{1,2,3\}$, $i \neq j$, the following formulas hold:
\begin{equation} \label{deux}
\frac{\pa_i f(v)}{f(v)} = \frac{v_j\,P_{ij}^f + v_i\, P_{ii}^f + \int_{\R^3} R_{ij}^f(v,w)\, f(w)
\, [w_i \, P_{ij}^f - w_j\, P_{ii}^f] \, dw}{(P_{ij}^f)^2 - P_{ii}^f\, P_{jj}^f},
\end{equation}
\begin{equation} \label{trois}
\frac{\pa_j f(v)}{f(v)} = \frac{v_i\,P_{ij}^f + v_j\, P_{jj}^f + \int_{\R^3} R_{ij}^f(v,w)\, f(w)
\, [w_i \, P_{jj}^f - w_j\, P_{ij}^f] \, dw}{(P_{ij}^f)^2 - P_{ii}^f\, P_{jj}^f},
\end{equation}
\begin{equation} \label{quatre}
v_i\, \frac{\pa_j f(v)}{f(v)} - v_j \, \frac{\pa_i f(v)}{f(v)} = \int_{\R^3} R_{ij}^f(v,w)\, f(w)\, dw.
\end{equation}

\end{Prop}

Note that thanks to the case of equality in Cauchy-Schwarz inequality, we know that 
$ (P_{ij}^f)^2 \neq P_{ii}^f\, P_{jj}^f$, so that formulas (\ref{deux}) and (\ref{trois}) are well defined.

\medskip

Proposition \ref{prop:Rij} can be seen as an inversion of formula (\ref{zero}).

\begin{proof}[Proof of Proposition \ref{prop:Rij}]
We consider $i,j \in \{1,2,3\}$ such that $i \neq j$.
Then, we expand $R_{ij}^f(v,w)$ in the following way:
\begin{equation} \label{dev}
R_{ij}^f(v,w)  = \bigg[v_i\, \frac{\pa_j f(v)}{f(v)} - v_j\, \frac{\pa_i f(v)}{f(v)} \bigg]
+ w_j\, \frac{\pa_i f(v)}{f(v)} - w_i\, \frac{\pa_j f(v)}{f(v)}
\end{equation}
$$ - v_i\, \frac{\pa_j f(w)}{f(w)} + v_j\, \frac{\pa_i f(w)}{f(w)} +
\bigg[w_i\, \frac{\pa_j f(w)}{f(w)} - w_j\, \frac{\pa_i f(w)}{f(w)} \bigg]. $$ 
Integrating (\ref{dev}) against $f(w)\, dw$, and recalling conditions (\ref{un}), we get
$$
\int_{\R^3} R_{ij}^f(v,w)\, f(w)\, dw = v_i\, \frac{\pa_j f(v)}{f(v)} - v_j \, \frac{\pa_i f(v)}{f(v)},
$$
which is exactly identical to (\ref{quatre}).

\medskip

Integrating then (\ref{dev}) against $f(w)\, w_i\, dw$, and recalling conditions (\ref{un}), we get
\begin{equation} \label{quinze}
\int_{\R^3} R_{ij}^f(v,w)\, f(w)\,w_i\, dw = - P_{ii}^f \, \frac{\pa_j f(v)}{f(v)} + P_{ij}^f \, \frac{\pa_i f(v)}{f(v)} - v_j.
\end{equation}
Finally, integrating (\ref{dev}) against $f(w)\, w_j\, dw$, and recalling conditions (\ref{un}) (or exchanging $i$ and $j$ in (\ref{quinze})), we get
\begin{equation} \label{seize}
\int_{\R^3} R_{ij}^f(v,w)\, f(w)\,w_j\, dw = - P_{ij}^f \, \frac{\pa_j f(v)}{f(v)} + P_{jj}^f \, \frac{\pa_i f(v)}{f(v)} + v_i.
\end{equation}

\par

Considering (\ref{quinze}), (\ref{seize}) as a $2 \times 2$ linear system with unknowns 
$\frac{\pa_i f(v)}{f(v)}$, $\frac{\pa_j f(v)}{f(v)}$, we get thanks to Cramer's formulas 
(recalling that $(P_{ij}^f)^2 \neq P_{ii}^f\, P_{jj}^f$ because of the case of equality in
Cauchy-Schwarz inequality) 
$$
\frac{\pa_i f(v)}{f(v)} = \frac{ Det\, \, 
 \left( \begin{array}{cc}
  v_j + \int_{\R^3} R_{ij}^f(v,w)\, f(w)\,w_i\, dw  &  - P_{ii}^f\\
  - v_i + \int_{\R^3} R_{ij}^f(v,w)\, f(w)\,w_j\, dw  &  - P_{ij}^f \end{array}\, \right)}
{ Det\, \, 
 \left( \begin{array}{cc}
  P_{ij}^f  &  - P_{ii}^f\\
  P_{jj}^f  &  - P_{ij}^f \end{array}\, \right)},
$$
$$
\frac{\pa_j f(v)}{f(v)} = \frac{ Det\, \, 
 \left( \begin{array}{cc}
  P_{ij}^f & v_j + \int_{\R^3} R_{ij}^f(v,w)\, f(w)\,w_i\, dw \\
  P_{jj}^f  & - v_i + \int_{\R^3} R_{ij}^f(v,w)\, f(w)\,w_j\, dw \end{array}\, \right)}
{ Det\, \, 
 \left( \begin{array}{cc}
  P_{ij}^f  &  - P_{ii}^f\\
  P_{jj}^f  &  - P_{ij}^f \end{array}\, \right)},
$$
which is exactly identical to formulas (\ref{deux}), (\ref{trois}).
\end{proof}

\bigskip

\begin{Prop}\label{prop:Delta}
We now define 
\begin{equation}\label{sept}
\Delta_f := \inf_{i,j = 1,2,3; i \neq j} \left( P_{ii}^f\, P_{jj}^f -  (P_{ij}^f)^2\right).
\end{equation}
Then there exists $C>0$ an explicitly computable constant number such that for all $f:= f(v) \ge 0$ satisfying (\ref{un}),
\begin{equation}\label{six}
\int_{\R^3} f(v) \,  \left| \frac{\nabla f(v)}{f(v)} + v \right|^2 \, \la v \ra^{-3} \, dv 
\end{equation}
$$\le C\, \Delta_f^{-2} \, \bigg( \sup_{i,j = 1,2,3; i \neq j} (P_{ij}^f)^2
+ \sup_{j = 1,2,3} |P_{jj}^f - 1|^2 + M_5(f)\, D(f) \bigg). $$

\end{Prop}

\begin{proof}[Proof of Proposition \ref{prop:Delta}]
Thanks to (\ref{deux}), we see that (for any $i,j \in \{1,2,3\}$ such that $i \neq j$)
$$
\frac{\pa_i f(v)}{f(v)} + v_i = v_j \, \frac{P_{ij}^f}{(P_{ij}^f)^2 - P_{ii}^f\, P_{jj}^f}
+ v_i \, \bigg( 1 + \frac{P_{ii}^f}{(P_{ij}^f)^2 - P_{ii}^f\, P_{jj}^f} \bigg)
$$
$$ + \frac{\int_{\R^3} R_{ij}^f(v,w)\, f(w)
\, [w_i \, P_{ij}^f - w_j\, P_{ii}^f] \, dw}{(P_{ij}^f)^2 - P_{ii}^f\, P_{jj}^f}, 
$$
so that, remembering that $P_{ij}^f \le 3/2$ for all $i,j \in \{1,2,3\}, i \neq j$, 
and $P_{ii}^f \le 3$ for all $i \in \{1,2,3\}$, since $ \sum_{i=1}^3 P_{ii}^f = 3$,
$$
\left|\frac{\pa_i f(v)}{f(v)} + v_i \right|^2  \le 3\, \Delta_f^{-2}\, \bigg(  |v_j|^2 \, (P_{ij}^f)^2 + |v_i|^2 \, \bigg| (P_{ij}^f)^2 + P_{ii}^f\, (1 - P_{jj}^f)  \bigg|^2
$$
$$
+ \, \bigg| \int_{\R^3} R_{ij}^f(v,w)\, f(w) \, [w_i \, P_{ij}^f - w_j\, P_{ii}^f] \, dw \bigg|^2 \bigg) $$
$$ \le 3\, \Delta_f^{-2}\, \bigg(  |v_j|^2 \, (P_{ij}^f)^2  
 + 2\, |v_i|^2 \, (P_{ij}^f)^4 + 18\, |v_i|^2 \, (1 - P_{jj}^f)^2 $$
$$
 + \,9\, \bigg[ \int_{\R^3} |R_{ij}^f(v,w)|\,  f(w) \, ( |w_i| + |w_j|) \, dw \bigg]^2 \bigg) . $$
Then,
$$
\int_{\R^3} f(v) \,  \left| \frac{\pa_i f(v)}{f(v)} + v_i \right|^2 \, \la v \ra^{-3} \, dv 
\le 3\, \Delta_f^{-2}\, \bigg( \frac{11}2\, (P_{ij}^f)^2  + 18\, (1 - P_{jj}^f)^2 
$$
 $$  +\, 9\,  \int_{\R^3} \la v \ra^{-3} \, f(v) \, \bigg\{ \int_{\R^3} |R_{ij}^f(v,w)|^2\,  f(w) 
 \, |v-w|^{-3} \, dw \bigg\}$$
$$ \times \, \bigg\{ \int_{\R^3}   f(w) 
 \, |v-w|^{3} \, ( |w_i| + |w_j|)^2 \, dw \bigg\}\, dv \bigg) $$
 $$ \le  \Delta_f^{-2}\, \left( \frac{33}2 \, (P_{ij}^f)^2  + 54\, (1 - P_{jj}^f)^2
 + {27}\, \int_{\R^3} \int_{\R^3}   f(v)\,  f(w)\, |R_{ij}^f(v,w)|^2\, |v-w|^{-3} \, dw dv 
\right.$$
 $$ \left. \times\, \sup_{v\in \R^3}  \la v \ra^{-3} \int_{\R^3}   f(w)\, (4 \, |v|^3 + 4\, |w|^3) \, 2\, |w|^2\, dw  \right). $$
Observing then (cf. \cite{D}, p.11-12) that
$$
D(f) = \frac14 \sum_{i=1}^3 \sum_{j=1}^3 \int_{\R^3} \int_{\R^3}   f(v)\,  f(w)\, |R_{ij}^f(v,w)|^2\, |v-w|^{-3} \, dw dv,
$$
we end up with the estimate
$$
\int_{\R^3} f(v) \,  \left| \frac{\nabla f(v)}{f(v)} + v \right|^2 \, \la v \ra^{-3} \, dv 
\le  \Delta_f^{-2}\, \bigg( \frac{99}2 \, \sup_{i,j \in \{1,2,3\}, i\neq j}(P_{ij}^f)^2 
$$
$$ + \, 162 \, \sup_{j \in \{1,2,3\} } (1 - P_{jj}^f)^2  + \,3456\, M_5(f)\, D(f) \bigg) , $$
so that (\ref{six}) holds with $C = 3456$.
\end{proof}

\bigskip

\begin{Prop}\label{prop:Deltabis}
One can find $C:= C(\bar{H})$ depending only on $\bar{H}$
such that for all $f\ge 0$ satisfying (\ref{un}), (\ref{unbis}),
the following inequalities hold:
\begin{equation}\label{dix}
 \Delta_f \ge C(\bar{H}),
 \end{equation}
\begin{equation}\label{onze}  
  \sup_{i,j \in \{1,2,3\}, i\neq j}(P_{ij}^f)^2 \le C(\bar{H})\, M_5(f)\, D(f),
 \end{equation}
\begin{equation}\label{douze}  
  \sup_{i,j \in \{1,2,3\}, i\neq j}|P_{ii}^f - P_{jj}^f|^2 \le C(\bar{H})\, M_5(f)\, D(f). 
 \end{equation}

\end{Prop}

\begin{proof}[Proof of Proposition \ref{prop:Deltabis}] 
We first observe that (thanks to \cite{D}, p.15), for any $\delta>0$,
$$
\Delta_f := \inf_{i,j = 1,2,3; i \neq j} \left( P_{ii}^f\, P_{jj}^f -  (P_{ij}^f)^2\right)
$$
 $$ \ge \inf_{|\theta |\le 1} \,\, \inf_{k=i,j;\, l=i,j;\,  k\neq l} \bigg(\int |v_k - \theta\, v_l |^2\, \,f(v)\,  dv \bigg)^2$$ 
$$ \ge \delta^4 \, \inf_{|\theta |\le 1} \,\, \inf_{k=i,j;\, l=i,j;\,  k\neq l} 
\bigg(\int_{|v_k - \theta\, v_l| \ge \delta, \, |v|\le \sqrt{6}} \,f(v)\, dv
\bigg)^2 $$
 $$ \ge  \delta^4 \, \bigg( 1 - \sup_{|\theta |\le 1} \,\, \sup_{k=i,j;\, l=i,j;\,  k\neq l} \int_{|v_k - \theta\, v_l| \le \delta, \, |v|\le \sqrt{6}} f(v)\, dv - \int_{ |v|\ge \sqrt{6} } f(v)\, dv \bigg)^2 $$
$$ \ge \delta^4 \,\bigg(   \frac{1}{2} - \sup_{|\theta |\le 1} \,\, \sup_{k=i,j;\, l=i,j;\,  k\neq l} \int_{ \frac{|v_k - \theta\, v_l|}{\sqrt{1+\theta^2}} \le \frac{\delta}{\sqrt{1+\theta^2}}, \, |v|\le \sqrt{6} } f(v)\, dv  \bigg)^2 $$
$$ \ge \delta^4 \,\bigg(   \frac{1}{2} - \sup_{|\theta |\le 1} \, \sup_{|A| \le 48 \, \frac{\delta}{\sqrt{1+\theta^2}} }\,  \int_{ A } f(v)\, dv  \bigg)^2 . $$
Using now the estimate (for all $q>0$ and $M>1$)
\begin{equation}\label{trentesix}
\sup_{|A| \le q} \int_A f(v)\, dv \le M\,q + \frac{\bar{H}}{\log M}, 
\end{equation}
we see that
$$
 \Delta_f \ge \delta^4 \,\bigg(\frac{1}{2} - 48\,\delta\,M - \frac{\bar{H}}{\log M} \bigg)^2,
$$
so that taking $M = e^{4\, \bar{H}}$ and $48\,\delta\,M = \frac18$ (that is $\delta = 2^{-7} \, 3^{-1} \, e^{-4 \, \bar H}$), we end up with
$$
 \Delta_f \ge \frac1{2^{34}}\, \frac1{3^4}\, e^{-16 \, \bar{H}},
$$
and estimate (\ref{dix}) is proven.
 
\bigskip

We now turn to the proof of estimates (\ref{onze}) and (\ref{douze}).

\medskip

Inserting (\ref{deux}) and (\ref{trois}) in (\ref{quatre}), we see that
$$
(v_i^2 - v_j^2) \, P_{ij}^f + v_i\,v_j\, (P_{jj}^f - P_{ii}^f)
$$
$$  + \int_{\R^3} R_{ij}^f(v,w)\, f(w)
\, \bigg[v_i \,w_i\,  P_{jj}^f - v_i\,w_j\, P_{ij}^f - v_j\,w_i\, P_{ij}^f + v_j\,w_j\, P_{ii}^f \bigg] \, dw $$
$$ = [ (P_{ij}^f)^2 - P_{ii}^f\, P_{jj}^f] \, \int_{\R^3} R_{ij}^f(v,w)\, f(w) \, dw , $$
and writing
$$
|v_i \,w_i\,  P_{jj}^f - v_i\,w_j\, P_{ij}^f - v_j\,w_i\, P_{ij}^f + v_j\,w_j\, P_{ii}^f| \le 9 \, |v| |w|
$$
it follows
$$
 |(v_i^2 - v_j^2) \, P_{ij}^f + v_i\,v_j\, (P_{jj}^f - P_{ii}^f)| \le   
\int_{\R^3} |R_{ij}^f(v,w)|\, f(w)
\,  \big[9 + 9\, |v| |w| \big]\, dw .
$$
Then 
$$
\int_{\R^3}  f(v) \, \la v \ra^{-5}\,  \bigg|(v_i^2 - v_j^2) \, P_{ij}^f + v_i\,v_j\, (P_{jj}^f - P_{ii}^f) \bigg|^2
\, dv 
$$
$$ 
 \le \int_{\R^3}  f(v) \, \la v \ra^{-5}\,  \bigg| \int_{\R^3} |R_{ij}^f(v,w)|\, f(w)
\, \big[9 + 9\, |v| |w| \big]\, dw \bigg|^2 \, dv $$
$$ 
\le \int_{\R^3}  f(v) \, \la v \ra^{-5}\,  \bigg\{  \int_{\R^3} |R_{ij}^f(v,w)|^2\, f(w)\, |v-w|^{-3}\, dw  \bigg\} $$
$$ 
\times\, \bigg\{  \int_{\R^3} f(w)\, |v-w|^{3}\, \big[9 + 9\, |v| |w| \big]^2 dw \bigg\}
\, dv $$
$$  
\le C \, D(f)\,  \sup_{v \in \R^3} \la v \ra^{-5}\,  \int_{\R^3} f(w)\, (|v|^3 + |w|^3) \, (1 +|v|^2 |w|^2)\, dw $$
$$  
\le C \, D(f)\,  \sup_{v \in \R^3} \la v \ra^{-5}\,  \int_{\R^3} f(w)\, \Big( |v|^3 + |v|^5 |w|^2 + |w|^3 + |v|^2 |w|^5 \Big)\, dw $$
$$  \le C\, D(f)\, M_5(f), $$
where $C>0$ is a (computable) constant number.
\medskip

We now observe that 
$$
\int_{\R^3}  f(v) \, { \la v \ra^{-5}} \,  \bigg|(v_i^2 - v_j^2) \, P_{ij}^f + v_i\,v_j\, (P_{jj}^f - P_{ii}^f) \bigg|^2 \, dv 
$$
$$ \ge S_f \, \bigg[ (P_{ij}^f)^2 + \frac14\, (P_{jj}^f - P_{ii}^f)^2 \bigg] , $$
where
$$
S_f := \inf_{\phi \in \R} \int_{\R^3}  f(v) \, { \la v \ra^{-5}}  \,  |(v_i^2 - v_j^2) \,\cos \phi +
2 \, v_i\,v_j\, \sin \phi|^2 \, dv .
$$
Introducing cylindrical coordinates defined by $v_i = r\, \cos\theta$,  $v_j = r\, \sin\theta$, and $v_k = z$ (where $k\neq i$ and $k\neq j$), we see that for all $\var>0$ (and assuming without loss of generality that $i=1$, $j=2$, $k=3$)
$$
S_f = \inf_{\phi \in \R} \int_{z\in\R}\int_{\theta=0}^{2\pi}\int_{r\in \R_+}
  f(r\, \cos\theta, r\, \sin\theta, z) \, (1 + r^2 + |z|^2)^{ -5/2}\, r^4\, |\cos (2\theta -\phi)|^2  \, r\, dr d\theta dz
$$
$$ \ge |\sin \var|^2 \bigg[ \int_{z\in\R}\int_{\theta=0}^{2\pi}\int_{r\in \R_+} 
f(r\, \cos\theta, r\, \sin\theta, z) \,r^4\, (1 + r^2 + |z|^2)^{ -5/2}  \, r\, dr d\theta dz $$
$$   - \sup_{\phi \in \R} \int_{z\in\R}\int_{\theta=0}^{2\pi}\int_{r\in \R_+} 
 f(r\, \cos\theta, r\, \sin\theta, z) \, {\mathbf 1}_{ \{ |2\theta - \phi - \frac\pi{2} \Z|\le \var \} } 
 \, r\, dr d\theta dz \bigg] . $$
As a consequence, denoting by $|\cdot|$ the Lebesgue measure on $\R^3$, for all $\var$, 
$\delta$, $R_1$, $R_2>0$, 
$$
S_f  \ge |\sin \var|^2 \bigg[ \int_{z\in\R}\int_{\theta=0}^{2\pi}\int_{r\in \R_+} 
 f(r\, \cos\theta, r\, \sin\theta, z) 
$$
$$ \times\, (1 + r^2 + |z|^2)^{ -5/2}\, r^4\, {\mathbf 1}_{ \{ r \ge \delta \} }\,
{\mathbf 1}_{ \{ r^2 + |z|^2 \le R_1^2 \} }   \, r\, dr d\theta dz $$
$$ - \sup_{\phi \in \R} \int_{z\in\R}\int_{\theta=0}^{2\pi}\int_{r\in \R_+} 
 f(r\, \cos\theta, r\, \sin\theta, z) \, {\mathbf 1}_{ \{ r^2 + |z|^2 \le R_2^2 \} }   \, {\mathbf 1}_{ \{ |2\theta - \phi - \frac\pi{2} \Z \} |\le \var } 
 \, r\, dr d\theta dz $$
 $$ - \int_{z\in\R}\int_{\theta=0}^{2\pi}\int_{r\in \R_+} 
 f(r\, \cos\theta, r\, \sin\theta, z) \, \, 
{\mathbf 1}_{ \{ r^2 + |z|^2 \ge R_2^2 \} }   \, r\, dr d\theta dz  \bigg]$$ 
$$ \ge |\sin \var|^2 \bigg[ \frac{\delta^4}{(1 + R_1^2)^{ 5/2}} \, \bigg(1 - \int_{v\in\R^3} f(v)\,
{\mathbf 1}_{ \{ |v_1|^2 + |v_2|^2 \le \delta^2 \} } \, {\mathbf 1}_{\{ |v|\le R_1 \} } \, dv $$
$$ - \int_{v\in\R^3} f(v)\, {\mathbf 1}_{ \{ |v|\ge R_1 \} } \, dv \bigg)
  - \int_{v\in\R^3} f(v)\, {\mathbf 1}_{\{|v|\ge R_2\} } \, dv
 - \sup_{|A| \le 16\,R_2^3\,\var} \int_A f(v)\, dv \,\bigg] $$
$$ \ge |\sin \var|^2 \bigg[ \frac{\delta^4}{(1 + R_1^2)^{ 5/2}} \, \bigg(1 -  \sup_{|A| \le 2\pi\,R_1\,\delta^2} \int_A f(v)\, dv - \frac{3}{R_1^2} \bigg) $$
$$ - \frac{3}{R_2^2}   - \sup_{|A| \le 16\,R_2^3\,\var} \int_A f(v)\, dv \,\bigg] .$$
Using now  estimate (\ref{trentesix}), we see that for all $\var$, $\delta$, $R_1$, $R_2$, $M_1$, $M_2>0$, 
$$
S_f  \ge |\sin \var|^2 \bigg[ \frac{\delta^4}{(1 + R_1^2)^{ 5/2}} \, \bigg(1 - 2\pi\,M_1\,R_1\,\delta^2 
- \frac{\bar{H}}{\log M_1} - \frac{3}{R_1^2} \bigg) - \frac{3}{R_2^2}   - 16\, M_2\,R_2^3\,\var
- \frac{\bar{H}}{\log M_2} \bigg]. 
$$
The proof of (\ref{onze}), (\ref{douze}) is concluded by selecting successively $R_1$ large enough, $M_1$ large enough, $\delta>0$ small enough, $R_2$ large enough, $M_2$ large enough, and $\var>0$ small enough.

\par

For example, selecting $R_1$, $M_1$ and $\delta$ in such a way that
$$
\frac{3}{R_1^2}  = \frac{\bar{H}}{\ln M_1} = 2\pi\,M_1\,R_1\,\delta^2 = \frac14, 
$$
we end up with
$$
S_f  \ge |\sin \var|^2 \bigg[ \frac{13^{ -5/2}}{48} \frac{e^{-8\, \bar{H} }}{64\,\pi^2} 
 - \frac{3}{R_2^2}   - 16\, M_2\,R_2^3\,\var
- \frac{\bar{H}}{\log M_2} \bigg]. 
$$
Then, we fix $R_2, M_2$ and $\var$ in such a way that
$$
\frac{3}{R_2^2}  = \frac{\bar{H}}{\log M_2} = 16\,M_2\,R_2^2\,\var 
= \frac{13^{ -5/2}}{48} \frac{e^{-8\, \bar{H} }}{256\,\pi^2}.
$$
\end{proof}

We can now conclude the proof of Theorem \ref{thm:entropy}.

\begin{proof}[Proof of Theorem \ref{thm:entropy}]
Thanks to Proposition \ref{prop:Delta} and Proposition \ref{prop:Deltabis}, we see that it only remains
to estimate $\sup_{j = 1,2,3} |P_{jj}^f - 1|$ in terms of $\sup_{i,j \in \{1,2,3\}, i\neq j}|P_{ii}^f - P_{jj}^f|$.

\medskip

In order to do so, we use the identity $ \sum_{j=1}^3 P_{jj}^f = 3$. We observe that
$$
|P_{11}^f - 1| = |P_{11}^f - \frac13\, \sum_{j=1}^3 P_{jj}^f | \le \frac13\, |P_{11}^f - P_{22}^f |
+  \frac13\, |P_{11}^f - P_{33}^f | . 
$$

Doing the same computation with $P_{22}^f$, $P_{33}^f$, we end up with
$$
 \sup_{j = 1,2,3}  |P_{jj}^f - 1|  \le \frac23\, \sup_{i,j \in \{1,2,3\}, i\neq j} |P_{ii}^f - P_{jj}^f |, 
$$
and Theorem \ref{thm:entropy} is proven.
\end{proof}
\medskip

We conclude this section with the proof of Corollary \ref{cerc}.
\medskip

\begin{proof}[Proof of Corollary \ref{cerc}]
Let $f\ge 0$ satisfy the normalization (\ref{un}), 
and define the weighted (with weight $\la v \ra^{\alpha}$ for $\alpha \in \R$)
relative Fisher information of $f$ with respect to $\mu$ (the centred reduced Gaussian) by
$$
I_\alpha (f|\mu) := \int_{\R^3} \frac{ |\nabla (f / \mu)|^2}{(f/ \mu)} \, \la v \ra^{\alpha} \, d\mu(v)
 = \int_{\R^3} f \,\left| \frac{\nabla f}{f} + v  \right|^2  \, \la v \ra^{\alpha} \, dv.
$$
With this notation, Theorem \ref{thm:entropy} becomes
\be\label{nt1}  
D(f) \ge C(\bar H)\, M_5(f)^{-1}\, I_{-3} (f|\mu).
\ee 
We define $\nu (v) := Z_1^{-1}\, \la v \ra^{-3} \,\mu(v)$ and $g(v) := Z_2^{-1} \la v \ra^{-3} f(v)$, where $Z_1 := \int \la v \ra^{-3}\, d\mu(v)$ and $Z_2 := \int \la v \ra^{-3} f(v) \, dv$ are normalization constants so that 
$\int \nu(v) \, dv = \int g(v) \, dv = 1$. Therefore we can rewrite
$$
I_{-3}(f| \mu) = \int \frac{ |\nabla (f / \mu)|^2}{(f/ \mu)} \, \la v \ra^{-3} \, d\mu(v) = Z_2  \int \frac{ |\nabla ( g / \nu)|^2}{( g/ \nu)} \, d\nu(v) =: Z_2 \, I(  g | \nu),
$$
where $I( g|\nu) := I_0( g|\nu)$ denotes the standard (i.e.\ without weight) relative Fisher information of $g$ with respect to $\nu$.
We now observe that $d\nu(v) = e^{-U(v)}\, dv$, with $U(v) = \frac{1}{2}|v|^2 + \frac{3}{2} \log (1+|v|^2) + U_0$,
 and where $U_0:= \log( (2\pi)^{3/2}\,Z_1 )$ is a normalization constant. We easily compute
$$
 \mathrm{Hess}\,\, U(v) = Id + \frac{3}{\la v \ra^2} \,Id - \frac{6}{\la v \ra^4} v \otimes v .
$$
Hence we obtain (for all $\xi, v \in \R^3$),
$$
\la \mathrm{Hess}\, U(v) \xi , \xi \ra = \left(1+\frac{3}{\la v \ra^2} \right)\,|\xi|^2 - \frac{6}{\la v \ra^4} (\xi \cdot v)^2 \ge \left(1 + \frac{3}{\la v \ra^2} - \frac{6\, |v|^2}{\la v \ra^4} \right) |\xi|^2 > \frac58\, |\xi|^2 .
$$
Indeed, for all  $z \in \R_+$ ($z=|v|^2$),
$$
1+ \frac{3}{1+z} - \frac{6 z}{ (1+z)^2} = \frac{z^2 -z +4}{(1+z)^2} \ge \frac{5}{8}.
$$
The probability measure $d\nu$ satisfies then a log-Sobolev inequality thanks to the Bakry-\'Emery criterion (see \cite{BE,AMTU}).
Therefore, for some $C>0$, 
$$
I( g | \nu ) \ge 
C\, \int   \frac{g}{\nu} \, \log \frac{g}{\nu} \, d\nu(v)
= C\, \int \left\{  \frac{g}{\nu} \, \log \frac{g}{\nu} +1 - \frac{g}{\nu} \right\} d\nu(v).
$$
Thanks to estimate (\ref{nt1}), we finally deduce, for some (new) constant $C(\bar H) >0$, 
$$
\ba
D(f) \ge C(\bar{H}) \, {  (M_5(f))^{-1} } \,  \int  \left\{   f
\log \left( \frac{Z_1}{Z_2} \frac{f}{\mu}  \right) + \frac{Z_2}{Z_1} \mu - f \right\} \la v \ra^{-3} \, dv,
\ea
$$
and the proof of \eqref{cc} is complete. Now, for any $R>0$, we estimate the integral from below remembering that $x \log x + 1 - x \ge 0$ for $x>0$,
$$
\ba
D(f) &\ge C(\bar{H}) \, {  (M_5(f))^{-1} } \, R^{-3} \,  \int_{\la v \ra \le R}  \left\{   f
\log \left( \frac{Z_1}{Z_2} \frac{f}{\mu}  \right) + \frac{Z_2}{Z_1} \mu - f \right\}  \, dv \\
&\ge C(\bar{H}) \, {  (M_5(f))^{-1} } \, R^{-3} 
\Bigg( \int_{\la v \ra \le R}  f \log (f / \mu) \, dv 
+\log (Z_1/Z_2)  \int_{\la v \ra \le R} f \, dv \\
&\qquad\qquad
+ \int_{\la v \ra \le R} (\mu - f) \, dv
+ (Z_2/Z_1 - 1) \int_{\la v \ra \le R} \mu \, dv \Bigg) .
\ea
$$
$$
\ba
&\ge C(\bar{H}) \, {  (M_5(f))^{-1} } \, R^{-3} 
\Bigg( \int  f \log (f / \mu) \, dv 
+\frac{Z_2}{Z_1} \left(  \frac{Z_1}{Z_2}\log (Z_1/Z_2) +1 - \frac{Z_1}{Z_2} \right) \\
&\qquad\qquad
- \int_{\la v \ra \ge R}  f \log (f / \mu) \, dv 
+\log (Z_2/Z_1)  \int_{\la v \ra \ge R} f \, dv \\
&\qquad\qquad
+ \int_{\la v \ra \ge R} (f - \mu) \, dv
+ (1-Z_2/Z_1) \int_{\la v \ra \ge R} \mu \, dv \Bigg) \\
&\ge C(\bar{H}) \, {  (M_5(f))^{-1} }\, R^{-3} 
\Bigg( \int  f \log (f / \mu) \, dv 
- \int_{\la v \ra \ge R}  f \log (f / \mu) \, dv \\
&\qquad\qquad\qquad
+(1+\log (Z_2/Z_1))  \int_{\la v \ra \ge R} f \, dv 
- (Z_2/Z_1) \int_{\la v \ra \ge R} \mu \, dv \Bigg) .
\ea
$$
Since $\int f = \int \mu=1$ and $\int \la v \ra^2 f = \int \la v \ra^2 \mu = 4$, we easily obtain that $2^{- 11/2} \le Z_1, Z_2 \le 1$. Then
$$
\ba
D(f) 
&\ge C(\bar{H}) \, {  (M_5(f))^{-1} }  \, R^{-3} 
\Bigg( \int  f \log (f / \mu) \, dv 
- \int_{\la v \ra \ge R}  f \log f  \, dv \\
&\qquad\qquad\qquad
-C  \int_{\la v \ra \ge R} \la v \ra^2 \, f \, dv 
- C \int_{\la v \ra \ge R} \mu \, dv \Bigg) ,
\ea
$$
which completes the proof of \eqref{eq:D<H}.
\end{proof}

\section{Moments estimates}\label{sec:moment}

In this section we prove estimates for the polynomial and exponential $L^1$-moments defined in \eqref{moment-poly} and \eqref{moment-exp}. For the sake of completeness we shall consider, only in this section, the Landau operator $Q$ (see \eqref{Q}) for general soft potentials, i.e.\ a matrix $a_{ij}$ (see \eqref{a}) for the whole range of soft potentials $-4 < \gamma < 0$. It is worth mentioning that, in the case $-2 \le \gamma \le 0$, estimates for polynomial moments have been established in \cite{Villani-BoltzmannBook,Wu,KC2} and stretched exponential moments in \cite{KC2}. Moreover, in the case $-4 < \gamma < -2$ polynomial moments estimates have been established in \cite{Villani-these,D}. We shall improve the above mentioned results in Lemma~\ref{lem:moment-poly-MS}, Corollary~\ref{cor:moment-exp-MS}, Lemma~\ref{lem:moment-poly-VS} and Corollary~\ref{cor:moment-exp-VS}.

\medskip

We begin with a key lemma on the coercivity of the collision operator in weighted $L^1$-space.
\par

\begin{lem}\label{lem:coercivity} 
Assume $-4 < \gamma < 0$. 
Let $f$ be a nonnegative function and  $\chi$ be either $\mathbf 1_{|\cdot|\le 1}$, or a smooth $ C^\infty_c(\R)$ radially symmetric cutoff function that satisfies $\mathbf 1_{B_{1/2}} \le \chi \le \mathbf 1_{B_1}$.  Let  $\chi_\eta (\cdot) = \chi(\cdot/\eta)$ with 
$\eta \in ]0,1]$, $l >2$ and 
$$ 
I= \iint_{\R^3 \times \R^3} f(v) f(w) \, |v-w|^{\gamma} \, \chi_\eta^c (v-w) \, \la v \ra^{l-2} \{ - \la v \ra^{2} + \la w \ra^2 \} \, dw \, dv ,
$$ 
where $\chi_\eta^c=1-\chi_\eta$.
\par
Then there exist constants $K,C>0$ such that
$$
\ba
I &\le - K\, M_0(f)^{1 - \gamma/2}\,M_2(f)^{\gamma/2} \, M_{l+\gamma}(f) 
+ C M_2(f) \, M_{l-2+\gamma}(f) \\
&\quad + C (M_2(f)/M_0(f))^{l/2-1 + \gamma} \, M_0(f) \, M_2(f) .
\ea
$$

\end{lem}

\begin{proof}
We decompose the integral into two parts $I = I_1+I_2$ with 
$$
\ba 
I_1 &= \iint_{\{|v-w| <  |w| \} \cap \{ |v-w| <  |v|\}} f(v) f(w) |v-w|^{\gamma} \chi_\eta^c(v-w) \la v \ra^{l-2} \{ - \la v \ra^{2} + \la w \ra^2 \} \, dw \, dv,\\
I_2 &= \iint_{\{|v-w| \ge  |w| \} \cup \{ |v-w| \ge  |v|\}} f(v) f(w) |v-w|^{\gamma} \chi_\eta^c(v-w)  \la v \ra^{l-2} \{ - \la v \ra^{2} + \la w \ra^2 \} \, dw \, dv. 
\ea
$$
For the first term $I_1$, we easily get $I_1 \le 0$ thanks to Young's
 inequality and using the symmetry of the region $\{|v-w| <  |w| \} \cap \{ |v-w| <  |v|\}$:
$$
\ba
& \iint_{\{|v-w| <  |w| \} \cap \{ |v-w| <  |v|\}} f(v) f(w) |v-w|^{\gamma} \chi_\eta^c(v-w)\, \la v \ra^{l-2} \la w \ra^2 \, dw \, dv \\
&\qquad 
\le 
\iint_{\{|v-w| <  |w| \} \cap \{ |v-w| <  |v|\}} f(v) f(w) |v-w|^{\gamma} \chi_\eta^c(v-w)\, \bigg[ \frac{l-2}l\, \la v \ra^{l}
+ \frac2l\, \la w \ra^l \bigg]\, dw \, dv \\
&\qquad
= \iint_{\{|v-w| <  |w| \} \cap \{ |v-w| <  |v|\}} f(v) f(w) |v-w|^{\gamma} \chi_\eta^c(v-w)\, \la v \ra^{l}  \, dw \, dv  .
\ea
$$
 
Next we observe that
$$
\ba
I_2 &=\iint_{\{|v-w| \ge  |w| \} \cup \{ |v-w| \ge  |v|\}} f(v) f(w) |v-w|^{\gamma} (\chi_\eta^c(v-w)- {\bf 1}_{|v-w| \ge 1}) \la v \ra^{l-2} \{ - \la v \ra^{2} + \la w \ra^2 \}  \, dw \, dv\\
&  \quad- \iint_{\{|v-w| \ge  |w| \} \cup \{ |v-w| \ge  |v|\}} f(v) f(w) |v-w|^{\gamma} \, {\bf 1}_{|v-w| \ge 1} \, \la v \ra^{l-2} \{ \la v \ra^2 - \la w \ra^2 \} \, dw \, dv.
\ea
$$
Using the estimate $ {\bf 1}_{|v-w| \ge 1} \le {\bf 1}_{|v-w|\ge \eta} \le \chi_\eta^c(v-w)$ since $\eta \in ]0,1]$ and following 
an argument similar to the one used for $I_1$, we obtain that the first term in the right-hand side of the
previous identity is nonpositive. Hence we have
$$
\ba
I_2 &\le -   \iint_{\{|v-w| \ge  |w| \} \cup \{ |v-w| \ge  |v|\}} f(v) f(w) |v-w|^{\gamma} \, {\bf 1}_{|v-w| \ge 1} \, \la v \ra^{l-2} \{ \la v \ra^2 - \la w \ra^2 \} \, dw \, dv
=: A + B,
\ea
$$
and we estimate each term separately.

For the term $A$, we first write that
$$
A  \le - \iint_{ \{|v-w| \ge  |w| \} } f(v) f(w) |v-w|^{\gamma} \, {\bf 1}_{|v-w| \ge 1} \, \la v \ra^{l} \, dw \, dv,
$$
and then we notice that the region $ \{|v-w| \ge  |w|\} \cap \{ |v-w|\ge 1 \}$ contains $ \{ |v| \ge 2 r  \} \cap \{ |w| \le r \}$ for any $r \ge 1$. Therefore, using that $-|v-w|^\gamma \le  - C \la w \ra^\gamma \la v \ra^\gamma $, 
$$
\ba
A &\le 
-  C \iint_{\{ |v| \ge 2 r  \} \cap \{ |w| \le r \} } f(v) f(w) |v-w|^{\gamma}  \la v \ra^{l} \, dw \, dv\\
&\le - C\left( \int_{ \{ |w| \le r \} }  \la w \ra^\gamma f(w) \, dw   \right)
\left( \int_{\{ |v| \ge 2 r  \} } \la v \ra^{l+\gamma} f(v) \,dv\right).
\ea
$$
We can easily compute
$$
\ba
\int_{ \{ |w| \le r \} }  \la w \ra^\gamma f(w) \, dw
&\ge \la r \ra^\gamma \int _{|w| \le r} f(w) \, dw \\
&= \la r \ra^\gamma \left(M_0(f)  - \int_{ \{ |w| \ge r \} }  f(w) \, dw  \right)
\ge \la r \ra^\gamma \left( M_0(f) -  \frac{M_2(f)}{\la r \ra^2} \right),
\ea
$$
and also
$$
\int_{\{ |v| \ge 2 r  \} }  \la v \ra^{l+\gamma} f(v) \, dv
= M_{l+\gamma}(f) - \int_{\{ |v| \le 2 r  \} }  \la v \ra^{l+\gamma} f(v) \, dv
\ge M_{l+\gamma}(f) - \la 2 r \ra^{l-2+\gamma} M_2(f).
$$
Gathering the previous estimates, we get
\be\label{eq:A}
A \le C \left( \frac{M_2(f)}{\la r \ra^2} - M_0(f) \right)\la r \ra^{\gamma} M_{l + \gamma}(f)
+ C \la 2r \ra^{l-2+2\gamma} M_0(f)\, M_2(f)
\ee
$$
\le - K \,M_0(f)^{1 - \gamma/2}\, M_2(f)^{\gamma/2}\, M_{l+\gamma}(f) + C (M_2(f)/M_0(f))^{l/2-1 + \gamma} \,M_0(f) M_2(f)
$$
by choosing 
$r$ such that $M_2(f)/\la r \ra^2 = M_0(f)/2$.

For the term $B$, we first decompose it into $B = B_1 + B_2 + B_3$, with
$$
\ba
B_1 &:=  \iint_{\{ |v-w| \ge  |v|\} \cap \{ |v-w|\ge 1 \}  } 
f(v) f(w) |v-w|^{\gamma}  \la v \ra^{l-2} \la w \ra^2  \, dw \, dv, \\
B_2 &:=  \iint_{\{ |w| \le |v-w| \le 2 |w| \} \cap \{ |v-w|\ge 1 \} } 
f(v) f(w) |v-w|^{\gamma}  \la v \ra^{l-2} \la w \ra^2  \, dw \, dv, \\
B_3 &:=  \iint_{\{|v-w| \ge  2 |w| \}  \cap \{ |v-w|\ge 1 \}  } 
f(v) f(w) |v-w|^{\gamma} \, \la v \ra^{l-2} \la w \ra^2  \, dw \, dv ,
\ea
$$
and we claim that
\be\label{Bj}
B_j \le C M_2(f)\, M_{l - 2 + \gamma}(f), \quad j=1,2,3.
\ee
Indeed, we first remark that for all the terms, we have $|v-w|^\gamma \le C\la v - w \ra^{\gamma}$ since $ |v-w| \ge 1 $, thus in order to prove \eqref{Bj}, we only need to prove that $|v-w| \ge c |v|$ for some constant $c>0$ in each case $j=1,2,3$. 
The first case $j=1$ is immediate since $|v-w| \ge  |v|$. We then observe that $|w| \le |v-w| \le 2 |w|$ implies $|v-w| \sim |w|$ and also $|v| \le C |w|$, which completes the case $j=2$. Finally, when $|v-w| \ge  2 |w|$, we obtain $|v-w| \sim |v|$ and the case $j=3$ also holds.

We get the desired result by patching together estimates \eqref{eq:A} and \eqref{Bj}.
\end{proof}

 We first state and prove estimates for $L^1$-moments in the moderately soft potentials case $-2 \le \gamma < 0$. We improve the results of \cite{Villani-BoltzmannBook, Wu, KC2}.

\begin{lem}\label{lem:moment-poly-MS}
Assume that $-2 \le \gamma <0$.
Let $f_0 \in L^1_2 \cap L \log L$ and consider any global $H-$ or weak solution $ f $ to the spatially homogeneous Landau equation (\ref{landau}) with initial data $f_0$. Suppose further that $f_0 \in L^1_l$ for some $l > 2$. Then, there exists a constant $C>0$ depending on $\gamma$, $M_0(f_0)$, $M_2(f_0)$ 
 (but not on $l$) such that, for all $t\ge 0$,
$$
M_l(f(t)) \le M_l(f(0)) + C \, l^{\frac{l+\gamma}{2}} \, t .
$$

\end{lem}

\begin{proof}
For simplicity we only give here the {\it{a priori}} estimates for the moments. The rigorous proof for any solution follows the same arguments as the ones that we shall present in Step 2 of the proof of Lemma~\ref{lem:moment-poly-VS} below, in the case of very soft potentials.

\medskip

Recall that thanks to the conservation of mass and energy,
 we have $M_0(f(t)) =M_0(f_0)$ and $M_2(f(t))= M_2(f_0)$ for all $t \ge 0$.
The equation for moments is (see e.g.\ \cite{DV1})
$$
\frac{d}{d t} M_l(f) =  \iint f(v) f(w) \, |v-w|^{\gamma} \, \la v \ra^{l} \left\{ - 2l + 2l \la v \ra^{-2}\la w \ra^2 + l(l-2) \la v \ra^{-4}[|v|^2 |w|^2 - (v\cdot w)^2]  \right\} \, dw \, dv.
$$
Because of the singularity of $|v-w|^{\gamma}$, we split it into two parts $|v-w|^{\gamma} \, {\bf 1}_{|v-w|\ge 1}\,$ and $|v-w|^{\gamma} \, {\bf 1}_{|v-w|\le 1}\,$, and we denote respectively $T_1$ and $T_2$ each term associated.

For the term $T_2$, we write
$$
\ba
T_2 &=  
l \iint_{\R^6} |v-w|^{\gamma}\, {\bf 1}_{|v-w|\le 1}\, \la v \ra^{l-2} \left\{ - 2 \la v \ra^2 + 2 \la w \ra^2   \right\}   f(v) f(w)  \, dw \, dv  \\
&\quad + l(l-2) \iint_{\R^6} |v-w|^{\gamma}\, {\bf 1}_{|v-w|\le 1}\, \la v \ra^{l-4} 
\left\{ |v|^2 |w|^2 - (v\cdot w)^2 \right\}   f(v) f(w)  \, dw \, dv \\
&=: T_{21} + T_{22}.
\ea
$$
Using Young's inequality, we easily obtain
$$
\ba
&\iint_{\R^6} f(v) f(w) |v-w|^{\gamma}\, {\bf 1}_{|v-w|\le 1}\, \la v \ra^{l-2} \la w \ra^2  \, dw \, dv  \\
&\qquad 
\le
\iint_{\R^6} f(v) f(w) |v-w|^{\gamma}\, {\bf 1}_{|v-w|\le 1}\, \la v \ra^{l}  \, dw \, dv ,
\ea
$$
and this implies $T_{21} \le 0$. Moreover, using the inequality $|v|^2 |w|^2 - (v\cdot w)^2 \le  |w|^2 |v-w|^2$, we get, since $\gamma+2 \ge 0$,
$$
\ba
T_{22} &\le C l^2 \iint_{\R^6} f(v) f(w) |v-w|^{\gamma+2} \, {\bf 1}_{|v-w|\le 1}\, \la v \ra^{l-4} \la w \ra^2  \, dw \, dv  \\
& \le C l^2 \, M_2(f) \, M_{l-4}(f) .
\ea
$$ 
We now investigate the term $T_1$, that we write
$$
\ba
T_1 &= -2l \iint_{\R^6}  |v-w|^{\gamma} \, {\bf 1}_{|v-w|\geq 1} \, \la v \ra^{l-2} \{ \la v \ra^2 - \la w \ra^2 \} f(v) f(w) \, dw \, dv  \\
&\quad
+ l(l-2) \iint_{\R^6}  |v-w|^{\gamma} \, {\bf 1}_{|v-w|\geq 1}\, \la v \ra^{l-4} \{|v|^2|w|^2 - (v \cdot w)^2  \} f(v)  f(w)  \, dw \, dv \\ 
&=: I + II.
\ea
$$
Thanks to 
Lemma \ref{lem:coercivity} (with $\chi = \mathbf 1_{|\cdot| \le 1}$), we already have 
$$
\ba
I &\le - K l \, M_0(f)^{1 - \gamma/2}\, M_2(f)^{\gamma/2} \, M_{l+\gamma}(f) 
+ C l \, M_2(f) \, M_{l-2+\gamma}(f) \\
&\quad  + C l \, (M_2(f)/M_0(f))^{l/2-1+\gamma} \, M_0(f) \, M_2(f).
\ea 
$$
We consider now the term $II$. If $l \le 4$, we easily observe that
$$
II \le C l^2 (M_2(f))^2.
$$
\color{black}
Now let $l>4$. We split $II = II_1 + II_2$, with
$$
II_1 = l(l-2) \iint_{ \{ |v-w|\geq 1 \} \cap \{ |w| \le  |v|\}} |v-w|^{\gamma} \, \la v \ra^{l-4} \{|v|^2|w|^2 - (v \cdot w)^2  \} f(v)  f(w)  \, dw \, dv , 
$$
and
$$
II_2 = l(l-2) \iint_{ \{ |v-w|\geq 1 \} \cap \{ |w| \ge  |v|\}} |v-w|^{\gamma} \, \la v \ra^{l-4} \{|v|^2|w|^2 - (v \cdot w)^2  \} f(v)  f(w)  \, dw \, dv. 
$$
Using the estimate $|v|^2|w|^2 - (v \cdot w)^2 \le |w|^2 |v-w|^2$, we get
$$
\ba
II_{1} 
&\le C l^2 \iint_{ \{ |v-w|\geq 1 \}  \cap \{ |w| \le  |v|  \}}
|v-w|^{\gamma+2} \la v \ra^{l-4} \la w \ra^{2} \, f(v)  f(w)  \, dw \, dv\\
&\le C l^2 \iint_{ \{ |v-w|\geq 1 \} \cap \{ |w| \le  |v|  \}} \la v \ra^{l-2 + \gamma} \la w \ra^{2} \, f(v)  f(w)  \, dw \, dv \\
& \le C l^2\, M_2(f)\, M_{l-2+\gamma}(f).
\ea
$$
Using now the inequality $|v|^2|w|^2 - (v \cdot w)^2 \le |v|^2 |v-w|^2$, it follows that
$$
\ba
II_{2} 
&\le C l^2 \iint_{ \{ |v-w|\geq 1 \}  \cap \{ |w| \ge  |v|  \}}
|v-w|^{\gamma+2} \la v \ra^{l-2}  \, f(v)  f(w)  \, dw \, dv \\
&\le C l^2 \iint_{ \{ |v-w|\geq 1 \}  \cap \{ |w| \ge  |v|  \}}  \la v \ra^{l-2} \la w \ra^{\gamma+2} \, f(v)  f(w)  \, dw \, dv \\
&\le C l^2 \iint_{ \{ |v-w|\geq 1 \}  \cap \{ |w| \ge  |v|  \}}  \la v \ra^{2} \la w \ra^{l-2+\gamma} \, f(v)  f(w)  \, dw \, dv \\
& \le C l^2 \,M_{2}(f)\, M_{l-2 + \gamma}(f).
\ea
$$

Gathering the previous estimates and recalling that $M_0(f)$ and $M_2(f)$ are constants, we get, for constants $K,C>0$,
$$ 
\ba
\frac{d}{dt} M_l (f)
&\le 
- K\, l\, M_{l+\gamma}(f) 
+ C l^2 \, M_{l-4}(f) 
+ C (l+l^2 \,{\mathbf 1_{l > 4}} ) \, M_{l-2 + \gamma}(f) + C (l + l^2 \, {\mathbf 1_{l \le 4}})  \\
&\le - K l \,M_{l+\gamma}(f)  + C l^2 \, M_{l-2 + \gamma}(f)  + C l^2,
\ea
$$
since $M_{l-4}(f) \le M_{l-2+\gamma}(f)$ (remember that $-2 \le \gamma < 0$).
If $l\le 4-\gamma $, then $M_{l-2 + \gamma}(f(t))$ is uniformly bounded and we easily get
$$
M_l(f(t)) \le M_l(f_0) + C\,t.
$$
Consider now $l > 4-\gamma$. Thanks to Young's inequality, for any $\epsilon>0$,
$$
M_{l-2 + \gamma }(f)
\le M_2^{\frac{2}{l-2+\gamma}}(f) \, M_{l+\gamma}^{ \frac{l-4+\gamma}{l-2+\gamma}}(f) 
\le C   \epsilon^{-\frac{l-4+\gamma}{2}}\,  M_2(f) + \epsilon \, M_{l+\gamma}(f).
$$
Hence it yields
$$
\frac{d}{d t} M_l(f) + K l \,M_{l+\gamma}(f)
\le C  l^2 \epsilon M_{l+\gamma}(f)
+  C l^2 \, \epsilon^{-\frac{l-4+\gamma}{2}} + C\,l .
$$
Choosing $\epsilon = \frac{K}{2C} l^{-1}$, we get
\be\label{eq:Mlbis}
\frac{d}{d t} M_l(f) + \frac{K}{2} l \,M_{l+\gamma}(f)
\le C (   l^{\frac{l+\gamma}{2}} + C\,l )
\le C l^{\frac{l+\gamma}{2}},
\ee
from which we deduce
$$
M_l(t) \le  M_l(f_0) + C l^{\frac{l+\gamma}{2} } \, t,
$$
which completes the proof.
\end{proof}

\begin{Cor}\label{cor:moment-exp-MS}
Consider the same setting as in Lemma~\ref{lem:moment-poly-MS}. 
Suppose further that $M_{s,\kappa}(f_0) = \int f_0(v)\, e^{\kappa \la v \ra^{s}}\,dv < \infty$ with $\kappa >0$ and $0< s <2$, or with $0< \kappa < 1/(2e)$ and $s=2$.
\par
Then, for some constant $C>0$ depending only on the parameters $\gamma, s, \kappa$ and the initial mass and energy
(that is, depending on $M_0(f_0)$, $M_2(f_0)$),

\begin{enumerate}[\quad(1)]

\item If $ s + \gamma < 0$, for all $t \ge 0$, 
$$
M_{s,\kappa}(f(t)) \le M_{s,\kappa}(f_0) + C\,t.
$$

\item If $ s+ \gamma \ge 0$, for all $t \ge 0$,
$$
M_{s,\kappa}(f(t)) \le M_{s,\kappa}(f_0) +  C .
$$

\end{enumerate}

\end{Cor}

\begin{rem}
As a direct consequence of Corollary \ref{cor:moment-exp-MS}-(2), the exponential convergence to equilibrium established in \cite[Theorem 1.4]{KC2}, for the case $-1 < \gamma < 0$, can be extended to the case
$-2 \le \gamma \le - 1$, as explained in \cite[Remark 1.5]{KC2}.
\end{rem}

\begin{proof} 
(1) We write $ e^{\kappa \la v \ra^s} = \sum_{j=0}^\infty \frac{\kappa^j}{j!} \la v \ra^{js}$ so that
$$
M_{s,\kappa}(f(t)) = \int \Big( \sum_{j=0}^\infty \frac{\kappa^j}{j!} \la v \ra^{js} \Big) \, f(t) \, dv = \sum_{j=0}^\infty \int   \frac{\kappa^j}{j!} \la v \ra^{js}  \, f(t) \, dv
= \sum_{j=0}^\infty \frac{\kappa^j}{j!} \, M_{js}(f(t)),
$$
where we have used Tonelli's theorem since the integrand in nonnegative (for any solution $f$). Thanks to Lemma~\ref{lem:moment-poly-MS}, we therefore obtain
$$
M_{s,\kappa}(f(t))  
\le  C t \sum_{j=0}^\infty \frac{\kappa^j}{j!} \, (sj)^{\frac{sj}{2} + \frac{\gamma}{2}} 
+ M_{s,\kappa}(f_0),
$$
and we only need to prove that the sum is finite. We rewrite
$$
\beta_j := \frac{\kappa^j}{j!} \, (sj)^{\frac{js}{2} + \frac{\gamma}{2}}
=  (\kappa \, s^{\frac{s}{2}} )^j  \, (sj)^{\frac{\gamma}{2}} \,  \frac{j^{\frac{sj}{2}}}{j!} , \quad j! \sim (j/e)^j\, \sqrt{2\pi\,j} \quad \text{as}\quad j \to \infty,
$$
hence we easily obtain that $\sum_{j=1}^\infty \beta_j < \infty$ for any $\kappa>0$ if $s < 2$, or for $0< \kappa  < 1/(2e)$ if $s=2$.

\smallskip
(2) As in the proof of Lemma~\ref{lem:moment-poly-MS}, we only give here the {\it{a priori}} estimates. Coming back to \eqref{eq:Mlbis}, one has
$$
\ba
\frac{d}{dt} M_{s,\kappa}(f)
&= \sum_{j=0}^\infty \frac{\kappa^j}{j!} \frac{d}{dt} M_{js}(f) \\
&\le 
\sum_{j=0}^\infty \frac{\kappa^j}{j!} \left\{  - K\,js\, M_{js + \gamma}(f) + C(js)^{\frac{js}{2} + \frac{\gamma}{2}}  \right\} \\
&\le 
-K \kappa s \sum_{j=1}^\infty \frac{\kappa^{j-1}}{(j-1)!}  M_{js + \gamma}(f) + C \sum_{j=0}^\infty \frac{\kappa^j}{j!}(js)^{\frac{js}{2} + \gamma/2}   \\
&= 
-K \kappa s \sum_{n=0}^\infty \frac{\kappa^{n}}{n!}  M_{ns + s + \gamma}(f) + C\sum_{j=0}^\infty \frac{\kappa^j}{j!}(js)^{\frac{js}{2} + \gamma/2}   =: I + II .
\ea
$$
Since $s+\gamma \ge 0$, we know that $- M_{ns + s +\gamma}(f) \le - M_{ns}(f)$,
 which implies the estimate $I \le - K \kappa \,s\, M_{s,\kappa}(f)$.
The second term $II$ is finite for any $\kappa>0$ if $s <2$, or for $0<\kappa<1/(2e)$ if $s=2$.
We finally obtain
$$
\frac{d}{dt} M_{s,\kappa}(f) \le - K \kappa s \, M_{s,\kappa}(f) + C,
$$
which implies the desired uniform in time bound.
\end{proof}

We now investigate the case of very soft potentials $-4 < \gamma \le - 2$. 
We get new estimates on the propagation of the  moments which improve the results of \cite[Appendix B, p.\ 193]{Villani-these} and \cite{D}.

\begin{lem}\label{lem:moment-poly-VS}
Assume that $-4 < \gamma \le -2$. 
Let $f_0 \in L^1_2 \cap L \log L (\R^3)$ and consider any global $H-$ or weak solution $f$ to the spatially homogeneous Landau equation (\ref{landau}) with initial data $f_0$. Assume moreover that
$f_0 \in L^1_l$ for some $l > 2$.
Then there exists a constant $C=C(\gamma, M_0(f_0), M_2(f_0), H(f_0)) >0$ (that does not depend on $l$) such that
$$
M_l(f(t)) \le C\, M_l(f_0) + C \, l^{(l-6)\frac{|\gamma+1|}{\gamma+4} - \gamma} \,t.
$$
\end{lem}

\begin{proof} 
We divide the proof into two steps. 

\smallskip

{\it Step 1: A priori estimates. }
We follow the argument of \cite[Appendix B, p.\ 193]{Villani-these} that uses the entropy formulation of solutions (cf. \cite{Vi}).
Let  $\eta=cl^{-1}$ and  $\chi_\eta (\cdot) = \chi(\cdot/\eta)$ where $c \in ]0,1/2[$ is a (small) constant and $\chi \in C^\infty_c(\R)$ is a smooth radially symmetric cutoff function such that $\mathbf 1_{B_{1/2}} \le \chi \le \mathbf 1_{B_1}$, as in Lemma \ref{lem:coercivity}. Recall that $a(z) = |z|^{\gamma+2} \Pi(z)$ 
and decompose $a = a_\eta + a^{c}_\eta$ with $a_\eta(z) := \chi_\eta(z)|z|^{\gamma+2} \Pi(z)$ and $a^{c}_\eta(z) := \chi_\eta^c(z) |z|^{\gamma+2} \Pi(z)$, where $\chi_\eta^c(z) := 1 - \chi_\eta(z) $.
We then write
$$
\frac{d}{dt} M_l(f) 
=I + II,
$$
where
\begin{equation}\label{nnst}
I =  - \frac12\,\iint f(v) f(w) \, \chi_\eta \, |v-w|^{\gamma+2} \, \Pi \left( \frac{\nabla f}{f}(v) - \frac{\nabla f}{f}(w)  \right) (\nabla \varphi(v) - \nabla \varphi(w)) \, dw \, dv,
\end{equation}
\begin{equation}\label{nnst2}
II = -\frac12 \,\iint f(v) f(w) \, \chi_\eta^c \, |v-w|^{\gamma+2} \, \Pi \left( \frac{\nabla f}{f}(v) - \frac{\nabla f}{f}(w)  \right) (\nabla \varphi(v) - \nabla \varphi(w)) \, dw \, dv,
\end{equation}
with $\Pi = \Pi(v-w)$, $\chi_\eta = \chi_\eta(v-w)$ and $\varphi(v) = \la v \ra^l$.
\par
Therefore, by Cauchy-Schwarz inequality and using that $\chi_\eta^2(v-w) \le {\mathbf 1}_{|v-w| \le \eta}$, it follows
$$
\ba
&|I| \le \left( \iint  f(v) f(w) \, |v-w|^{\gamma+2} \, \Pi \left( \frac{\nabla f}{f}(v) - \frac{\nabla f}{f}(w)  \right)\left( \frac{\nabla f}{f}(v) - \frac{\nabla f}{f}(w)  \right) dw \, dv  \right)^{1/2} \\
&\qquad\times
\left( \iint f(v) f(w) |v-w|^{\gamma+2}\, {\mathbf 1}_{|v-w| \le \eta} \,  \Pi (\nabla \varphi(v) - \nabla \varphi(w)) (\nabla \varphi(v) - \nabla \varphi(w)) \, dw \, dv  \right)^{1/2}, 
\ea
$$
and the first integral is bounded by $\sqrt{2D(f)}$, see \eqref{D(f)}. For $\varphi(v) = \la v \ra^{l}$, we have
$$
\ba
|v-w|^2  \Pi (\nabla \varphi(v) - \nabla \varphi(w)) (\nabla \varphi(v) - \nabla \varphi(w))
&= l^2 \left[ |v|^2 |w|^2 - (v \cdot w)^2 \right] \left[ \la v \ra^{l - 2} - \la w \ra^{l - 2} \right]^2 .
\ea
$$
Using  the estimate $$ |v|^2 |w|^2 - (v \cdot w)^2  \le \min (|v|^2, |w|^2) \, |v-w|^2 , $$
and $ |\la v \ra^{l - 2} - \la w \ra^{l - 2} | \le l \max (\la v \ra^{\l-3}, \la w \ra^{l-3}) \, |v-w|$, we finally get
$$
\ba
|v-w|^2 \Pi (\nabla \varphi(v) - \nabla \varphi(w)) (\nabla \varphi(v) - \nabla \varphi(w))
&\le C  l^4 \{  \la v \ra^{2l-4} + \la w \ra^{2l-4} \} \, |v-w|^4 .
\ea
$$
Then 
$$
|I| \le C l^2  D(f)^{1/2} \left( \iint f(v) f(w) |v-w|^{\gamma+4} \, {\bf 1}_{|v-w|\le \eta} \,   \la v \ra^{2l-4}   \, dw \, dv\right)^{1/2}.
$$
Since the last integral is over $\{ |v-w| \le \eta =cl^{-1} \}$, we claim that
  there exist  universal constants $C_i(i=1,2)$ such that $C_1\la v \ra^{l} \le \la w \ra^{l} \le C_2\la v \ra^{l}$. 
Indeed, we have
\ben\label{equivalent1}
\la v\ra^l&=&(1+|v|^2)^{l/2}\le \big(\la w\ra^2+2cl^{-1}(|w|+cl^{-1})\big)^{l/2}\notag
\\&\le& \sum_{k=0}^{[(l+2)/4]} \frac{\frac{l}2(\frac{l}2-1)\dots (\frac{l}2-k+1)}{k!} \bigg[ \la w\ra^{2k}(2cl^{-1})^{l/2-k}(|w|+cl^{-1})^{l/2-k}\notag\\&&+\,
(2cl^{-1})^k(|w|+cl^{-1})^k\la w\ra^{l-2k}\bigg]\notag\\
&\le& \sum_{k=0}^{[(l+2)/4]} \frac{1}{k!} \la w\ra^l\le C\la w\ra^l,
\een where we use the fractional binomial expansion. Using the symmetry of $v$ and $w$, we conclude
 the proof of the claim. 
Now we obtain
\be\label{eq:I}
\ba
|I| &\le C l^2  D(f)^{1/2} \left( \iint f(v) f(w) |v-w|^{\gamma+4} \, {\bf 1}_{|v-w|\le \eta} \,    \la w \ra^{l-4} \la v \ra^{l} \, dw \, dv  \right)^{1/2} \\
&\le C l^2  D(f)^{1/2} \, \eta^{\gamma/2+2} \, M_l^{1/2}(f) \, M_{l-4}^{1/2}(f)\\
&\le  D(f) M_l(f) + C l^{-\gamma}  \,M_{l-4}(f).
\ea
\ee
For the term $II$, we use the usual weak formulation (cf.\ \cite{Vi}), obtained from (\ref{nnst2}) by performing an integration 
by parts w.r.t. both $v$ and $w$:
$$
II = \iint f(v) f(w) \Big\{  a^{c}_\eta(v-w):\nabla^2 \varphi(v)  + 2 b^c_\eta(v-w) \cdot \nabla \varphi (v)    \Big\} \, dw \, dv ,
$$
where $(b^c_\eta(z))_{i=1,2,3} =  \sum_{j=1}^3 \partial_j(a^{c}_\eta(z))_{ij} = \chi^c_\eta(|z|) b_{i}(z) + \sum_{j=1}^3 (\partial_j[\chi^c_\eta(|z|)] ) \, a_{ij}(z)$. It follows then
$$
\ba
II &= \iint f(v) f(w) \, \chi_\eta^c(|v-w|) \Big\{  a(v-w):\nabla^2 \varphi(v)  + 2 b(v-w) \cdot \nabla \varphi(v)    \Big\}  \, dw \, dv \\
&\quad
+ 2 \sum_{i=1}^3 \sum_{j=1}^3 \iint f(v) f(w) \, (\partial_j[\chi^c_\eta(|v-w|)] ) \, a_{ij}(v-w) \, \partial_i \varphi(v) \,  dw \, dv,
\ea
$$
and we remark that the second integral vanishes since, for $j\in \{ 1,2,3 \}$, $(\partial_j[\chi^c_\eta(|z|)] ) = (\chi_\eta^c)'(|z|) \,  |z|^{-1} z_j$ and $\sum_{j=1}^3 a_{ij}(z) z_j = 0$. We finally obtain (see the proof of Lemma~\ref{lem:moment-poly-MS}) 
$$
\ba
II 
&= 2l \iint f(v) f(w) |v-w|^{\gamma} \chi_\eta^c(v-w) \, \la v \ra^{l-2} \{ - \la v \ra^{2} + \la w \ra^2 \}  \,  dw \, dv \\
&\quad
+ l(l-2) \iint f(v) f(w) |v-w|^{\gamma} \chi_\eta^c(v-w) \, \la v \ra^{l-4} \{ |v|^2|w|^2 - (v \cdot w)^2 \}  \,  dw \, dv \\
&=: II_1 + II_2.
\ea
$$
Using the elementary inequality $ |v|^2|w|^2 - (v \cdot w)^2 \le |w|^2 |v-w|^2$, we easily obtain, thanks to 
the estimate $\chi_\eta^c(v-w) \le {\bf 1}_{|v-w|\ge \eta/2}$,
\be\label{eq:II2}
II_2 \le l^2 \iint_{\R^6} f(v) f(w) |v-w|^{\gamma+2} \, {\bf 1}_{|v-w|\ge \eta/2} \, \la w \ra^2 \la v \ra^{l-4} \,  dw \, dv 
\le C l^{-\gamma} \,  M_2(f) \, M_{l-4}(f).
\ee
The term $II_1$ is controlled thanks to Lemma \ref{lem:coercivity} (with $\chi$ a $C^{\infty}$ function), which gives 
\ben\label{eq:II1}
II_1 &\le &  - K \,l \, M_0(f)^{1-\gamma/2}\,M_2(f)^{\gamma/2} \, M_{l+\gamma}(f) + C l \, M_2(f) \, M_{l-2+\gamma}(f)
 \notag\\&&+ \, C l \, (M_2(f)/M_0(f))^{l/2-1+\gamma} \, M_0(f) \, M_2(f).
\een

Finally, gathering  estimates \eqref{eq:I}, \eqref{eq:II2} and \eqref{eq:II1},  and using that $M_0(f)$ and $M_2(f)$ are constant in time, it follows  
\be\label{eq:Ml0}
\frac{d}{dt} 
M_l(f) + K l\, M_{l + \gamma}(f) \le  D(f)\, M_l(f) + C l^{-\gamma}\,  M_{l-4}(f) 
+ C l \, M_{l-2+\gamma}(f) + C\, l.
\ee
Notice that $M_{l-2+\gamma}(f) \le M_{l-4}(f)$ because $-4 < \gamma \le -2$,
 thus the third term in the right-hand side of \eqref{eq:Ml0} can be absorbed into the second one.
We also recall that $\int_0^\infty D(f(s)) \, ds \le C <\infty$ for some positive constant $C$ depending on $H(f_0)$.

If $2 < l \le 6$, then $\sup_{t\ge0} M_{l-4}(f(t)) \le C$, hence
$$
\frac{d}{dt} M_l(f) + K l \, M_{l + \gamma}(f) \le  D(f) \, M_l(f) + C,
$$
and by Gronwall's lemma,
$$
M_l(f(t)) \le C M_l(f_0) + C t.
$$
Suppose now that $l > 6$. We use Young's inequality to obtain, for any $\epsilon>0$,
$$
M_{l-4}(f) 
\le C \epsilon^{- \frac{l-6}{\gamma+4}} \,M_2(f) + \epsilon \,M_{l+\gamma}(f).
$$
Coming back to \eqref{eq:Ml0} and choosing $\epsilon = \frac{K}{2C}\,  l^{1+\gamma}$, we then get
$$
\ba
\frac{d}{dt} M_l(f)  + \frac{K}{2} l \, M_{l+\gamma}(f) 
&\le D(f)\, M_l(f) + C l^{(l-6)\frac{|\gamma+1|}{\gamma+4} - \gamma} + C\, l \\
&\le D(f) \, M_l(f) + C l^{(l-6)\frac{|\gamma+1|}{\gamma+4} - \gamma} ,
\ea
$$
hence
$$
M_l(f(t)) \le C M_l(f_0) + C l^{(l-6)\frac{|\gamma+1|}{\gamma+4} - \gamma}\, t,
$$
which yields the desired result.

\smallskip

{\it Step 2: Rigorous proof.} Let $W^l_\delta(v)= \langle v\rangle^l(1+\delta |v|^2)^{-\frac{l}2} $ with   $\delta \in ]0,1/2[$ being a (small) parameter,
 and set $M_l^\delta(f)=\int f(v) W^l_\delta(v) \, dv$. It follows that
$W^l_\delta\in W^{2,\infty} (\R^3)$ and then it can be chosen as a test function in the formulation of the weak solution, that is,
$$\ba \int_{\R^3} f(t)W^l_\delta\, dv= \int_{\R^3} f_0W^l_\delta\, dv +\int_0^t\int_{\R^3} Q(f, f) W^l_\delta\, dv \, dt .\ea$$
Similarly to \eqref{nnst} and \eqref{nnst2}, we have 
$$\ba \int_{\R^3} Q(f, f) W^l_\delta\, dv=I_\delta+II_{\delta}, \ea $$
where \begin{equation}\label{nnstd}
I _\delta=  - \frac12\,\iint f(v) f(w) \, \chi_\eta \, |v-w|^{\gamma+2} \, \Pi \left( \frac{\nabla f}{f}(v) - \frac{\nabla f}{f}(w)  \right) (\nabla W^l_\delta(v) - \nabla W^l_\delta(w)) \, dw \, dv,
\end{equation} and
\ben\label{nnstd2}
II_\delta& =& -\frac12 \,\iint f(v) f(w) \, \chi_\eta^c \, |v-w|^{\gamma+2} \, \Pi \left( \frac{\nabla f}{f}(v) - \frac{\nabla f}{f}(w)  \right) (\nabla W^l_\delta(v) - \nabla W^l_\delta(w)) \, dw \, dv\notag\\
&=& \iint f(v) f(w) \, \chi_\eta^c  \,  \big(a(v-w): \nabla^2 W^l_\delta(v)
+2b(v-w) \cdot \nabla W^l_\delta(v)\big) \, dw \, dv
\een
with $\Pi = \Pi(v-w)$ and $\chi_\eta = \chi_\eta(v-w)$.

\smallskip

{\it Estimate of $I_\delta$.}  By following the estimate of $I$ in {\it Step 1}, 
we  first have 
$$
\ba
&|I_\delta| \le \left( \iint  f(v) f(w) \, |v-w|^{\gamma+2} \, \Pi \left( \frac{\nabla f}{f}(v) - \frac{\nabla f}{f}(w)  \right)\left( \frac{\nabla f}{f}(v) - \frac{\nabla f}{f}(w)  \right) dw \, dv  \right)^{1/2} \\
&\qquad\times
\left( \iint f(v) f(w) |v-w|^{\gamma+2}\,  \chi_\eta \,  \Pi (\nabla W^l_\delta(v) - \nabla W^l_\delta(w)) (\nabla W^l_\delta(v) - \nabla W^l_\delta(w)) \, dw \, dv  \right)^{1/2}. 
\ea
$$
We claim that
\beno  |I_\delta| \le  Cl^2\sqrt{2D(f)}\bigg(\iint f(v)f(w)  \chi_\eta \, |v-w|^{\gamma+4}W^l_\delta(v)W^l_\delta(w) \langle w\rangle^{-4} \, dw \, dv\bigg)^{\frac12}.\eeno
Indeed, following the computation of \eqref{equivalent1}, we first have
 \beno \mathrm{1}_{|v-w|\le cl^{-1}}  (1+\delta|v|^2)^{\frac{l}2}\sim  \mathrm{1}_{|v-w|\le cl^{-1}}   (1+\delta|w|^2)^{\frac{l}2}, \eeno
from which, together with the fact $\mathrm{1}_{|v-w|\le cl^{-1}} \langle v\rangle^{l}\sim\mathrm{1}_{|v-w|\le cl^{-1}} \langle w\rangle^{l}$, we get
the (uniform w.r.t. $\delta \in ]0,1/2[$) estimate
  \ben\label{equivalent2} \mathrm{1}_{|v-w|\le cl^{-1}} W^l_\delta(v)\sim \mathrm{1}_{|v-w|\le cl^{-1}} W^l_\delta(w). \een
  Next by the mean value theorem, we see that
\beno &&\mathrm{1}_{|v-w|\le cl^{-1}} |\nabla W^l_\delta(v) - \nabla W^l_\delta(w)|\\&&\le \mathrm{1}_{|v-w|\le cl^{-1}}
\int_0^1 |(\nabla^2 W^l_\delta) (v+t(w-v))|dt |v-w|\\
&&\le  \mathrm{1}_{|v-w|\le cl^{-1}}
Cl^2\int_0^1  W^l_\delta(v+t(w-v) )  \langle v+t(w-v) \rangle^{-2} dt\, |v-w|,\eeno
where we have used the estimate $|\nabla^2_v W^l_\delta (v)|\le Cl^2 W^l_\delta(v)\langle v\rangle^{-2}$ which is a consequence of \eqref{dwldelta} (see below).
Notice that $|(v+t(w-v))-v|\le cl^{-1}$ and $|(v+t(w-v))-w|\le cl^{-1}$. Then by \eqref{equivalent2} we obtain
that, if $|v-w|\le cl^{-1}$, $$ W^l_\delta(v+t(w-v))\sim  W^l_\delta(v) \sim   W^l_\delta(w).  $$
Thus
\beno
&&\mathrm{1}_{|v-w|\le cl^{-1}} |\nabla W^l_\delta(v) - \nabla W^l_\delta(w)|\\&&\le \mathrm{1}_{|v-w|\le cl^{-1}}
Cl^2|v-w|\min\{  W^l_\delta(v),   W^l_\delta(w)\}\min\{\langle v\rangle^{-2},\langle w\rangle^{-2}\}, 
 \eeno
which is enough to prove the claim.

As a consequence,
\beno   |I_\delta| 
&\le& C l^2  D(f)^{1/2} \, \eta^{\gamma/2+2} \, (M_l^\delta(f) )^{1/2} \, \bigg(\int f(v)W^l_\delta(v) \langle v\rangle^{-4} dv\bigg)^{1/2}\\
&\le& D(f)M_l^\delta(f) + Cl^{-\gamma}\bigg(\int f(v)W^l_\delta(v) \langle v\rangle^{-4} dv\bigg).\eeno

\smallskip

{\it Estimate of $II_\delta$.}  Inserting the computations 
\ben \partial_{j} W^l_\delta (v) &=&lW^l_\delta(v) \frac{v_j}{\langle v\rangle^{2}}-lW^l_\delta(v) \frac{\delta v_j}{1+\delta|v|^2},\notag
\\ \partial_{ij} W^l_\delta(v) &=& lW^l_\delta(v) \frac{\delta_{ij}}{\langle v\rangle^{2}}+l(l-2)W^l_\delta(v) \frac{v_iv_j}{\langle v\rangle^{4}}-lW^l_\delta(v)\frac{\delta\delta_{ij}}{1+\delta|v|^2}\notag\\&&+l(l+2)W^l_\delta(v)\frac{\delta^2 v_iv_j}{(1+\delta|v|^2)^2}-2l^2 W^l_\delta(v) \frac{\delta }{1+\delta|v|^2} \frac{ v_iv_j}{\la v \ra^2},
 \label{dwldelta}\een
into \eqref{nnstd2}, we have
\be\label{IIdelta}  
\ba
II_\delta &= l\iint  f(v) f(w) \, \chi_\eta^c \, |v-w|^{\gamma} W^l_\delta(v) \langle v\rangle^{-2}\bigg(-2|v|^2+2|w|^2+(l-2) \frac{|v|^2|w|^2-(v\cdot w)^2}{\langle v\rangle^2}\bigg) dvdw
\\&+l\iint  f(v) f(w) \, \chi_\eta^c \, |v-w|^{\gamma} W^l_\delta(v) \frac{\delta}{ 1+\delta|v|^2}\bigg(2|v|^2-2|w|^2+(l+2)  (|v|^2|w|^2-(v\cdot w)^2) \frac{\delta}{1+\delta|v|^2} \bigg) dvdw\\
&-2l^2\iint  f(v) f(w) \, \chi_\eta^c \, |v-w|^{\gamma} W^l_\delta(v) \frac{\delta\langle v\rangle^{-2}}{ 1+\delta|v|^2}\big(  |v|^2|w|^2-(v\cdot w)^2  \big) dvdw.
\ea
\ee

It is obvious that
\beno  |II_\delta|\le C l^2 \iint f(v)f(w) W^l_\delta(v) (1+\langle w\rangle^2)dvdw.\eeno
Then thanks to the estimates of $I_\delta$ and $II_\delta$, we get
\beno  M_l^\delta(f(t))-M_l^\delta(f_0)\le \int_0^t D(f(s))M_l^\delta(f(s)) ds+Cl^2(M_2(f_0)+1)\int_0^t M_l^\delta(f(s)) ds. \eeno
By Gronwall's inequality, we obtain that, for some constant $C>0$ and for all $t\ge0$,
\beno M_l^\delta(f(t))\le C M_{l}^\delta(f_0) + C e^{C l^2 \, t} , \eeno
which first implies that $M_l(f)$ is bounded (uniformly locally in time):
\ben\label{first bound} M_l(f(t)) \le C M_{l}(f_0) + C e^{C l^2 \,  t}, \een 
thanks to Fatou's Lemma. We shall now use the bound \eqref{first bound} in order to improve the moment estimate. We recall that
\beno  M_l^\delta(f(t))-M_l^\delta(f_0)&=& \int_0^t I_\delta(s)ds+\int_0^t II_\delta(s) ds
\\&\le& \int_0^t D(f(s))M_l^\delta(f(s)) ds+Cl^{-\gamma}\int_0^t  M_{l-4}(f(s))ds+\int_0^t II_\delta(s) ds.\eeno
Then by Gronwall's inequality, 
we get
\be \label{preq}
M_l^\delta(f(t))\le CM_l(f_0)+Cl^{-\gamma}\int_0^t  M_{l-4}(f(s))ds+C\int_0^t II_\delta(s) ds. \ee
Thanks to \eqref{IIdelta} and the bound \eqref{first bound},
 we easily observe that we can apply the dominated convergence theorem for the last term in the right-hand side of eq. (\ref{preq}).
 We therefore obtain, for all $t \ge 0$,
\beno M_l(f(t))\le CM_l(f_0)+Cl^{-\gamma}\int_0^t  M_{l-4}(f(s))ds+C\int_0^t II(s) ds.  \eeno
From that inequality we can copy the argument used in the first step to give the bound on the term $II$ and thus obtain the desired estimate. This ends the proof of the lemma.
\end{proof}

\begin{Cor}\label{cor:moment-exp-VS}
Consider the same setting as in Lemma~\ref{lem:moment-poly-VS}.
Suppose further that $M_{s,\kappa}(f_0) = \int f_0(v)\, e^{\kappa \la v \ra^{s}}\, dv < \infty$ 
with $\kappa >0$ and $0<s< \frac{\gamma+4}{|\gamma+1|}$, or with $0< \kappa < \frac{1}{e}\frac{|\gamma+1|}{\gamma+4}$ and $s =\frac{\gamma+4}{|\gamma+1|}$. Then there is $C>0$ depending on the parameters $\gamma$, $s$, $\kappa$ and the initial
mass energy and entropy (that is, depending on $M_0(f)$, $M_2(f)$ and $H(f_0)$) such that
$$
M_{s,\kappa}(f(t))   \le C\, M_{s,\kappa}(f(0))+ C\,t, \quad \forall \, t \ge 0.
$$

\end{Cor}

\begin{proof} 
The proof follows the same arguments as in the proof of Corollary~\ref{cor:moment-exp-MS}-(1). From Lemma~\ref{lem:moment-poly-VS} and writing $e^{\kappa \la v \ra^s} = \sum_{j=0}^\infty \kappa^j \frac{\la v \ra^{js}}{j!}$ we get
$$
M_{s, \kappa} (f(t)) = \sum_{j=0}^\infty \frac{\kappa^j}{j!} \, M_{js}(f(t)) \le C t \sum_{j=0}^\infty \frac{\kappa^j}{j!} \, (sj)^{(sj-6)\frac{|\gamma+1|}{\gamma+4} - \gamma} + C M_{s,\kappa} (f_0).
$$
We then conclude by observing that
$$
\beta_j := \frac{\kappa^j}{j!} \, (sj)^{(sj-6)\frac{|\gamma+1|}{\gamma+4} - \gamma}
= (\kappa s^{s\frac{|\gamma+1|}{\gamma+4}})^{j}  \, (sj)^{-6\frac{|\gamma+1|}{\gamma+4} - \gamma} \, \frac{j^{js\frac{|\gamma+1|}{\gamma+4}}}{j!},
\quad j! \sim (j/e)^j\, \sqrt{2\pi\,j} \quad \text{as}\quad j \to \infty,
$$
which implies $\sum_{j=0}^\infty \beta_j < \infty$ under the assumptions of the corollary.
\end{proof}

\begin{rem}\label{rem:cor:moment-exp-VS}
If we consider $s + \gamma \ge 0$, the same argument presented in the proof of 
Corollary~\ref{cor:moment-exp-MS}-(2) would give us a
 uniform in time bound for the moment $M_{s,\alpha}(f(t))$.
 But the conditions $\gamma \in (-4, -2]$, $0< s \le (\gamma+4)/|\gamma+1|$ 
and $s + \gamma \ge 0$ imply $\gamma=-2$ and $s=2$, so that we recover exactly the result stated in Corollary~\ref{cor:moment-exp-MS}-(2) (for $\gamma=-2$).

\end{rem}

\section{Large time behaviour}

We now turn to the proof of Theorem \ref{thm:decay}. Before starting it, we state an interpolation lemma.

\begin{lem}\label{Halpha}

(i) Let $r \in ]1,3[$, 
 and $\alpha\in \R$.
Define $\theta(r,\alpha) = \frac{9(r-1) + 2 \alpha}{3-r}$. Then, there exists a constant $C:= C(r)>0$, such that
for any $f: = f(v)\ge 0$,
$$
\int_{\R^3} \la v \ra^{\alpha}\, f\, |\log f|\, dv \le C\, \bigg(  M_{\alpha+2}(f) + M_{\theta(r,\alpha)}(f)^{\frac{3-r}{2}} \, \| f \|_{L^3_{-3}(\R^3)}^{\frac32(r-1)} + 1 \bigg).
$$

\smallskip

(ii) Let $r \in ]1,3[$, $s \in ]0,2[$ and $\kappa>0$. Then for any $\kappa_1 > \kappa$ and $\kappa_2 > 2\kappa/(3 - r)$,
one can find a constant $C:= C(r, \kappa, \kappa_1, \kappa_2)>0$ such that for any $f: = f(v)\ge 0$,
$$
\int_{\R^3} e^{\kappa \la v \ra^s}\, f\, |\log f|\, dv \le C\, \bigg(   M_{s,\kappa_1}(f) + M_{s,\kappa_2}(f)^{\frac{3-r}{2}} \, \| f \|_{L^3_{-3}}^{\frac32(r-1)} + 1 \bigg).
$$

\end{lem}

\begin{proof}
(i) We decompose the integral into
$$
\int \la v \ra^{\alpha} f |\log f| = \int \la v \ra^{\alpha} f |\log f| \, \{ \mathbf 1_{f >1} + \mathbf 1_{e^{-|v|^2} < f \le 1} + \mathbf 1_{0 \le f \le e^{-|v|^2}} \} =: I_1 + I_2 + I_3 .
$$
For the term $I_1$, we notice that $ f\, |\log f| \, \mathbf 1_{f >1}\le C(r)\, f^r \, \mathbf 1_{f >1} $ (for $r \in ]1,3[$, and some constant $C(r) >0$).
\par
Thanks to H\"older's inequality, we get
$$
\ba
I_1 &= \int \la v \ra^{\alpha} f \log f \, \mathbf 1_{f >1} \\
&\le C(r)\, \int \la v \ra^\alpha f^r  = C(r)\,\int \la v \ra^{\alpha + \frac92(r-1)} \, f^{\frac12 (3-r)}\,\, \la v \ra^{-\frac92(r-1)} \, f^{\frac32 \,(r-1)} \\
& \le C(r)\, M_{\theta(r,\alpha)}(f)^{\frac{3-r}{2}} \, \| f \|_{L^3_{-3}}^{\frac32(r-1)}.
\ea
$$
For $I_2$,  we use the inequality $|\log f|\, \mathbf 1_{e^{-|v|^2} < f \le 1} 
\le \la v \ra^2\, \mathbf 1_{e^{-|v|^2} < f \le 1}$, in order to obtain
$$
I_2 \le M_{\alpha+2}(f).
$$
Finally, for $I_3$, we use the estimate
 $f\, | \log f|\, \mathbf 1_{0 \le f \le e^{-|v|^2}}  \le C\, \sqrt f\,   \mathbf 1_{0 \le f \le e^{-|v|^2}}\le 
C\, e^{-|v|^2/2}$, so that
$$
I_3 \le C\, \int \la v \ra^\alpha e^{-|v|^2/2} \le C,
$$
for some constant $C>0$.
\smallskip

(ii) We decompose the integral into
$$
\int e^{\kappa \la v \ra^s}\, f\, |\log f| = \int e^{\kappa \la v \ra^s}\, f \, |\log f| \, \{ \mathbf 1_{f >1} + \mathbf 1_{e^{-|v|^2} < f \le 1} + \mathbf 1_{0 \le f \le e^{-|v|^2}} \} =: I'_1 + I'_2 + I'_3 .
$$
For the term $I'_1$, again thanks to H\"older's inequality, we obtain (for any $r \in ]1,3[$ and some constant $C(r)>0$)
$$
\ba
I'_1 &= \int e^{\kappa \la v \ra^s}\, f \,|\log f| \,  \mathbf 1_{f >1} \\
&\le C(r)\, \int e^{\kappa \la v \ra^s}\, f^r
= C(r)\,\int \la v \ra^{\frac92(r-1)}\,  e^{\kappa \la v \ra^s}\, f^{(3-r)/2}\,\, 
  \la v \ra^{-\frac92(r-1)} \, f^{\frac32 \,(r-1)} \\
&\le C(r, \kappa, \kappa_2)\, M_{s,\kappa_2}(f)^{\frac{3-r}{2}} \, \| f \|_{L^3_{-3}}^{\frac32\,(r-1)}.
\ea
$$
For $I'_2$, we get (with constants whose dependence is explicitly stated)
$$ I'_2 \le C(\kappa, \kappa_1)\, M_{s, \kappa_1}(f).
$$
Finally, for $I'_3$, we get
$$
I'_3 \le C(\kappa)\, \int e^{\kappa\, \la v\ra^{s} -|v|^2/2} \le C(\kappa).
$$
\end{proof}

From now on we consider a global (equivalently $H$- or weak) solution $f:= f(t,v) \ge 0$ to the spatially homogeneous Landau equation
with Coulomb potential (\ref{landau}), associated to nonnegative initial data $f_0 \in L^1_2 \cap L \log L (\R^3)$ satisfying the normalization \eqref{f0}. 

%
%
%

We shall use in the sequel the following properties of such a solution: 
\begin{itemize}

\item the conservation of mass, momentum and energy (\ref{conssh}), more precisely
$$
\int f(t,v) \,dv = 1 , \quad
\int f(t,v)\, v\, dv = 0, \quad 
\int f(t,v) \, |v|^2 \, dv = 3 , \quad \forall\, t \ge 0 ;
$$

\item the moments estimates of Lemma~\ref{lem:moment-poly-VS} and Corollary~\ref{cor:moment-exp-VS};

\item the entropy-entropy dissipation inequality
\be\label{boundHD}
H(f(t)) + \int_0^t D(f(\tau)) \, d\tau \le H(f_0) \le C_0 , \quad \forall \, t \ge 0,
\ee

\item $f$ satisfies estimate \eqref{L3<D}.

\end{itemize}

\medskip

We now start the:

\begin{proof}[Proof of Theorem \ref{thm:decay}-(i)]

Recall that we assume that $f_0 \in  L^1_{\ell} (\R^3)$ with 
$  \ell > \frac{19}{2}$, which can be
rewritten as $\ell = \frac{9}{2} + \frac{3k}{2}$ and $ k > \frac{10}{3} $.
\medskip

We split the proof into four steps. All constants $C>0$ in the proof of this part of the theorem 
are simply written by $C$ (and can change from line to line),
 but in fact depend on $\ell, M_\ell(f_0), \bar{H}$, where $\int f_0\,\log f_0 \, dv \le \bar{H}$ (if $f_0$ were not normalized, they also would depend on the mass and energy of $f_0$).   

\medskip\noindent
{\it Step 1.}
First of all, we obtain from Lemma~\ref{lem:moment-poly-VS} (with $\gamma=-3$) that $M_\ell(f(t)) \le C(1+t)$, and interpolating this result with the conservation of the energy $M_2(f(t))=M_2(f_0) = 4$ for any $t \ge 0$, it follows
$$   
M_5(f(t)) \le C\, (1+t)^{\frac{3}{\ell-2}},
$$
since $ \ell > 5$.

\medskip\noindent
{\it Step 2.}
As a consequence we can write, using the entropy dissipation inequality \eqref{eq:D<H} of Corollary~\ref{cerc} with some $R(t)>0$ to be chosen later (note that $C$ depends on $k$ below),
$$
\ba
D(f(t))
&\ge  C\, (1+t)^{ - \frac{3}{\ell-2}} \,R(t)^{-3}\,  H(f(t) | \mu) \\
&\quad
- C\, (1+t)^{  - \frac{3}{\ell-2}} \,R(t)^{-3-k}\, \left\{   
\, \int_{\R^3} f(t)\, |\log  f(t)| \, \la v \ra^{k} \,dv
+ M_{k+2}(f(t)) + 1  \right\}.
\ea
$$
Thanks to Lemma~\ref{Halpha}-(i) with $\alpha=k$, $r=5/3$, so that $\frac32 (r-1) =1$ and $\theta(r,k) = \frac{9+3k}{2}$, we get
$$
\ba
D(f(t))
&\ge   C\,(1+t)^{ - \frac{3}{\ell-2}}\, R(t)^{-3}\, H(f(t)|\mu) \\
&\quad
- C\, (1+t)^{  - \frac{3}{\ell-2}}\, R(t)^{-3-k}\, \bigg( M_{k+2}(f(t)) + M_{\frac{9+3k}{2}}(f(t))^{\frac23} \, \| f(t) \|_{L^3_{-3}} +1 \bigg) .
\ea
$$
Thanks to \eqref{L3<D},  we know that
$$
M_{\frac{9+3k}{2}}(f(t))^{\frac{2}{3}} \, \| f(t) \|_{L^3_{-3}} 
\le C\, M_{\frac{9+3k}{2}}(f(t))^{\frac{2}{3}} \, (  D(f(t)) + 1 ),
$$
thus it follows
$$
\ba
& \Big(1+  C\,(1+t)^{  - \frac{3}{\ell-2}}\, M_{\frac{9+3k}{2}}(f(t))^{\frac{2}{3}}\, R(t)^{-3-k} \Big)\, D(f(t)) \\
&\qquad
 \ge  C\,(1+t)^{ - \frac{3}{\ell-2}}\, R(t)^{-3} H(f(t)|\mu) \\
&\qquad\quad  - C\,(1+t)^{ - \frac{3}{\ell-2}}\, R(t)^{-3-k}\,  
M_{k+2}(f(t)) \\
&\qquad\quad  - C\, (1+t)^{  - \frac{3}{\ell-2}} \,R(t)^{-3-k} \, M_{\frac{9+3k}{2}}(f(t))^{\frac{2}{3}}    .
\ea
$$

\medskip\noindent
{\it Step 3.}
We now choose $R(t) = (1+t)^\nu$ for some $  \nu \in ]0, \frac13 \frac{(3k-1)}{(3k+5)}[ $ such that $  \nu\,k> \frac23 $.
 Note that this is possible whenever 
$  \frac13 \,\frac{(3k-1)}{(3k+5)} \,k > \frac23 $, which is implied
by the condition $  k> \frac{10}{3} $.
\smallskip

 Using Lemma~\ref{lem:moment-poly-VS} (with $\gamma=-3$) again and interpolating the estimate with the conservation of the energy as in step 1, we have (since $\frac{9 + 3k}2 = \ell$)
$$
M_{k+2}(f(t)) \le C\, (1+t)^{\frac{k}{\ell-2}}, \quad
[M_{\frac{9+3k}{2}}(f(t))]^{\frac{2}{3}} \le C\, (1+t)^{\frac23}.
$$
Therefore, noticing that
$$
C\,(1+t)^{ - \frac{3}{\ell-2}} \, M_{\frac{9+3k}{2}}(f)^{\frac{2}{3}}\, R(t)^{-3-k} \le C,
$$
we end up with
$$
\ba
 D(f(t)) 
& \ge  \frac{C\, H(f(t)|\mu) }{(1+t)^{ 3\nu + \frac{6}{5+3k}}} 
-\frac{C}{(1+t)^{ 3\nu+k\nu + \frac{6}{5+3k} - \frac{2k}{5+3k}}}
-\frac{C}{(1+t)^{ 3\nu+k\nu + \frac{6}{5+3k} - \frac23}} \\
& \ge  \frac{C\, H(f(t)|\mu) }{(1+t)^{ 3\nu + \frac{6}{5+3k}}} 
-\frac{C}{(1+t)^{ 3\nu+k\nu + \frac{6}{5+3k} - \frac23}}\, .
\ea
$$

\medskip\noindent
{\it Step 4.}
Denoting $x(t) := H(f(t) | \mu)$, $  a := 3\nu + \frac{6}{5+3k}$ and 
$  b := 3\nu + \nu k + \frac{6}{5+3k} - \frac23$, we plug the last estimate into \eqref{boundHD} and obtain the following inequality (we denote by $C_1>0$ and $C_2>0$ two different constants in order to 
avoid confusions)
\be\label{eq:ode1}
x(t) + C_1 \int_0^t \frac{x(\tau)}{(1+\tau)^a} \, d \tau \le C_0 + C_2 \int_0^t (1+\tau)^{-b} \, d\tau,
\ee
with $0<a<1$ (because $   \nu < \frac13\, \frac{3k-1}{3k+5} $) and $b>a$ 
(because $  \frac13 \,\frac{(3k-1)}{(3k+5)} \,k > \frac23$,
since  $  k> \frac{10}{3}$).

We recall that the generalized Gronwall inequality (see e.g.\ \cite{TosVi2})
$$ 
u(t) \le \phi(t) + \int_0^t \lambda(\tau) u(\tau) d\tau
$$
implies
$$
u(t) \le \phi(0) e^{\int_0^t \lambda(\tau) d\tau }
+ \int_0^t e^{\int_\tau^t \lambda(\sigma) d\sigma} \, \frac{d\phi(\tau)}{d\tau} \, d\tau,
$$
that we apply to \eqref{eq:ode1} and obtain
$$
x(t) \le C_0 \, e^{- C_1\,\frac{(1+t)^{1-a} }{1-a} } 
+ C_2 \, e^{- C_1\,\frac{(1+t)^{1-a} }{1-a} } 
\int_0^t (1+\tau)^{-b}\,
e^{C_1 \,\frac{(1+\tau)^{1-a} }{1-a} }\, d\tau .
$$

Then we observe that thanks to an integration by parts,
$$ e^{- C_1\,\frac{(1+t)^{1-a}}{1-a} } \, \int_0^t (1+\tau)^{-b}\,
e^{C_1 \, \frac{(1+\tau)^{1-a}}{1-a} }\, d\tau $$
$$ =  e^{- C_1\,\frac{(1+t)^{1-a}}{1-a} } \, \bigg[
C_1^{-1}\, e^{C_1\,\frac{(1+t)^{1-a}}{1-a} }\, (1+t)^{a-b} - C_1^{-1}\, e^{\frac{1}{1-a} }
+ C_1^{-1}\,(b-a)\,  \int_0^t (1+\sigma)^{a-b-1}\, 
e^{C_1\,\frac{(1+\sigma)^{1-a}}{1-a} }\, d\sigma \bigg] $$
$$ \le C_1^{-1}\, (1+t)^{a-b} + C_1^{-1}\, (b-a)\,e^{- C_1\,\frac{(1+t)^{1-a}}{1-a} } \, 
\int_0^t (1+\sigma)^{a-b-1}\, 
e^{C_1\,\frac{(1+\sigma)^{1-a}}{1-a} }\, d\sigma $$
$$ \le  C_1^{-1}\, (1+t)^{a-b} +  C_1^{-1}\, (b-a)\, \frac{t}2\, e^{- C_1\,\frac{(1+t)^{1-a}}{1-a} } \, 
(1+ t/2)^{a-b-1}\, e^{C_1\,\frac{(1+ t/2)^{1-a}}{1-a} } + 
 C_1^{-1}\, (b-a)\, \frac{t}2\, (1+ t/2)^{a-b-1}.$$
We therefore see that
$$
x(t) \le C\, (1+t)^{- ( b-a)} = C\,(1+t)^{- \left( \nu k - \frac23    \right)}.
$$
This entails that $H(f(t)|\mu) \le C\, (1+t)^{- \beta}$ for all $ \beta \in ]0, \frac{k}3\, \frac{(3k-1)}{(3k+5)} - \frac23[$,
 and this
ends the proof of  Theorem \ref{thm:decay}-(i).
\end{proof}

We now turn to the 
\begin{proof}[Proof of Theorem \ref{thm:decay}-(ii)]
In the sequel, the constant denoted by $C$ in fact depends on $\kappa$, $s$, and $\bar{H}$.
 We split the proof into the same four steps as in the proof of Theorem \ref{thm:decay}-(i).

\medskip\noindent
{\it Step 1.}
From Corollary~\ref{cor:moment-exp-VS}, we get the estimate
$$
M_{s,\kappa}(f(t)) \le C\,(1+t).
$$
Interpolating this estimate with the conservation of the energy $M_2(f(t)) = M_2(f_0)=4$ (for all $t \ge 0$), we claim that 
$$ 
M_5(f(t)) \le C\,\log^{3/s} (1+t) .
$$
Indeed we can write, for any $r > (6/(\kappa\,s))^{1/s}$, 
$$ 
\ba
\int \la v \ra^{ 5} f(t) 
&= \int_{\la v \ra\le r} \la v \ra^{ 5}\, f(t) + \int_{ \la v \ra > r} \la v \ra^{ 5} \, f(t) \\
&\le r^{ 3}\, M_2(f(t)) + r^{ 3} \,e^{- \frac{\kappa}2\, r^s} \int_{ \la v \ra > r}
 \la v \ra^{2}\,
 e^{ \frac{\kappa}2\, \la v \ra^s} f(t) \\
&\le C\, r^{ 3} \, \big( 1 +  e^{- \frac{\kappa}2\, r^s}\, M_{s,\kappa}(f(t)) \big).
\ea
$$
We choose then $r = \sup[ (\frac{2}{\kappa}\, \log M_{s,\kappa}(f(t))\, )^{1/s}, (6/(\kappa\,s))^{1/s}] $, 
which concludes the claim.

\medskip\noindent
{\it Step 2.}
Using the entropy dissipation inequality \eqref{eq:D<H} of Corollary~\ref{cerc}, we can argue as in the proof of Theorem~\ref{thm:decay}-(i) above and obtain, for any $\kappa_0 \in ]0,\kappa[$
and some $R(t)>0$ to be chosen later:
$$
\ba
D(f(t)) 
&\ge  C \log^{ -3/s}(1+t) \, R(t)^{-3} \,  
H(f(t) | \mu)  \\
&\quad 
- C \log^{ -3/s}(1+t) \, R(t)^{-3} \, e^{-\kappa_0 \, R(t)^s} \bigg\{ 
\int e^{\kappa_0\, \la v \ra^s}\, f(t) \,| \log f(t)|
+\int \la v \ra^2 e^{\kappa_0\, \la v \ra^s}\, f(t) + 1  \bigg\} .
\ea
$$%
Then, for $r \in ]1,3[$, $\kappa_1 > \kappa_0$, $\kappa_2 > 2\kappa_0/(3-r)$, it follows from Lemma~\ref{Halpha} and the bound \eqref{L3<D} that
$$
\ba
D(f(t))
&\ge  C\, \log^{ -3/s}(1+t) \, R(t)^{-3} \, H(f(t) | \mu)\\
&\quad - C\, \log^{ -3/s}(1+t) \, R(t)^{-3} \,  e^{-\kappa_0\, R(t)^s} \, \bigg[
2\, M_{s, \kappa_1}(f(t)) + 1 + M_{s, \kappa_2}(f(t))^{(3-r)/2} \,||f(t)||_{L^3_{-3}}^{\frac32(r-1)} \bigg].
\ea
$$
As a consequence, considering $r = 5/3$, and $\kappa_1 \in ]\kappa_0, \kappa [$, $\kappa_2 > \frac32\,\kappa_0$, we get 
$$
\ba
& D(f(t)) \, \bigg(1 + C\,\log^{ -3/s}(1+t) \, R(t)^{-3} \, 
 e^{-\kappa_0 \,R(t)^s}\, M_{s, \kappa_2}(f(t))^{2/3} \bigg) \\
&\qquad 
\ge  \frac{ C\, H(f(t) | \mu)}{R(t)^3 \, \log^{ 3/s}(1+t)}
- \frac{C\,(1+t)}{e^{\kappa_0\, R(t)^s} \, R(t)^3 \, \log^{ 3/s}(1+t)}  .
\ea
$$%

\medskip\noindent
{\it Step 3.}
We now choose $R(t) =  (1+t)^{\frac{1}{3+s}}\,  (\log(1+t))^{ -\frac{3 + qs}{3s}}$, for some $q\in \R$ to be chosen later, so that, taking $\kappa_0 \in ]0, \frac23\,\kappa [$ and 
$\kappa_2 \in ] \frac32\,\kappa_0, \kappa [$,
$$ 
C \log^{ -3/s}(1+t) \, R(t)^{-3} \,  e^{-\kappa_0 R(t)^s}\, M_{s, \kappa_2}(f(t))^{2/3} \le C . 
$$
Denoting $x(t) := H(f(t) | \mu)$ and gathering the previous estimates together with \eqref{boundHD}, we see that (denoting by $C_1,C_2>0$ the constants to avoid confusions)
$$
\ba
& x(t) + C_1 \int_0^t \frac{x(\tau)}{(1+\tau)^{\frac{3}{3+s} } \, 
(\log(1+\tau))^{-q } } \, d\tau  \\
&\qquad \le C_0 + C_2 \int_0^t e^{- \kappa_0 (1+\tau)^{\frac{s}{3+s}} \,( \log(1+\tau))^{ -\frac{3+qs}{3}}} \, (1+\tau)^{\frac{s}{3+s} }\, (\log(1+\tau))^{q } \, d\tau ,
\ea
$$
thus by the generalized Gronwall's inequality it follows, denoting $A(t) := \int_0^t (1+\tau)^{-\frac{3}{3+s}} \, (\log(1+\tau))^{q } \, d\tau$,
$$
x(t) \le C_0 \, e^{-C_1 A(t)} + C_2\,  e^{-C_1 A(t)} \int_0^t e^{C_1 \, A(\tau)} \, (1+\tau)^{\frac{s}{3+s}}\,
 (\log(1+\tau))^{q} \, e^{-\kappa_0 (1+\tau)^{\frac{s}{3+s}} \, 
 (\log(1+\tau))^{ -\frac{3+qs}{3}}} \, d\tau.
$$
We can now complete the proof by some elementary computations.
We observe that (thanks to an integration by parts)
for all $a> - 1$ and $b\in \R$,
$$ \int_0^t (1+\sigma)^a\, (\log(1+\sigma))^b \, d\sigma = \frac{(1+ t)^{a+1}}{a+1}\,
(\log(1+t))^b + O \big( (1+t)^{a+1}\, (\log(1+t))^{b-1} \big) , $$
so that (for $t$ large enough)
\be\label{bound-int} 
2\,(1+t)^{\frac{s}{3+s}}\,  
 (\log(1+t))^{q } \ge \int_0^t (1+\sigma)^{-\frac{3}{3+s}} \,
 (\log(1+\sigma))^{q } \, d\sigma \ge (1+t)^{\frac{s}{3+s}}\,  
 (\log(1+t))^{q } . 
\ee
We hence split
$$
\ba
x(t) &\le C_0 \, e^{-C_1 \, A(t)}   \\
&\quad 
+ C_2\, e^{-C_1 \, A(t)} \int_{0}^{t/n} e^{C_1 \, A(\tau)} \, (1+\tau)^{\frac{s}{3+s}}\,
 (\log(1+\tau))^{q} \, e^{-\kappa_0 (1+\tau)^{\frac{s}{3+s}} \, 
 (\log(1+\tau))^{ -\frac{3+qs}{3}}} \, d\tau \\
&\quad 
+ C_2\, e^{-C_1 \, A(t)} \int_{t/n}^{t} e^{C_1 \, A(\tau)} \, (1+\tau)^{\frac{s}{3+s}}\,
 (\log(1+\tau))^{q} \, e^{-\kappa_0 (1+\tau)^{\frac{s}{3+s}} \, 
 (\log(1+\tau))^{ -\frac{3+qs}{3}}} \, d\tau \\
& =: e^{-C_1 \, A(t)} \, x(0) + I_1 + I_2,
\ea
$$
for some $n>0$ to be chosen large enough. Thanks to \eqref{bound-int}, we easily get
$$
C_0 \, e^{-C_1 \, A(t)}  \le C_0 \, e^{- C_1\,(1+t)^{\frac{s}{3+s}}\,  
 (\log(1+t))^{q } } .
$$
Moreover we write for the term $I_1$, using \eqref{bound-int},
$$
\ba
I_1 
&\le C_2 \, e^{-C_1 (A(t) - A(t/n))} \, \int_{0}^{t/n}  (1+\tau)^{\frac{s}{3+s}}\,
 (\log(1+\tau))^{q} \, e^{-\kappa_0 (1+\tau)^{\frac{s}{3+s}} \, 
 (\log(1+\tau))^{ -\frac{3+qs}{3}}} \, d\tau   \\
&\le C \, e^{-C_1 (A(t) - A(t/n))} 
\le C \, e^{-C_1 (1+t)^{\frac{s}{3+s}} (\log (1+t))^q + 2 C_1 (1+t/n)^{\frac{3}{3+s}} (\log(1+t/n))^q} \\
&\le C \, e^{-C_1 \, (1+t)^{\frac{s}{3+s}} \, (\log(1+t))^q \, ( 1 - 2 n^{- \frac{3}{(3+s)} } ) }
\le C \, e^{- \frac{C_1}2 \, (1+t)^{\frac{s}{3+s}} \, (\log(1+t))^q  },
\ea
$$
when $n>0$ is chosen large enough (i.e. such that $ 1 - \frac{2}{n^{3/(3+s)}}  \ge 1/2$).
For the other term, we have
$$
\ba
I_2 &\le C_2 \, \int_{t/n}^{t}  (1+\tau)^{\frac{s}{3+s}}\,
 (\log(1+\tau))^{q} \, e^{-\kappa_0 (1+\tau)^{\frac{s}{3+s}} \,
  (\log(1+\tau))^{ -\frac{3+qs}{3}}} \, d\tau \\
&\le C \, e^{-\frac{\kappa_0}{2} (1+t/n)^{\frac{s}{3+s}} \,
 (\log(1+t/n))^{ -\frac{3+qs}{3}}} \int_{t/n}^{t}  (1+\tau)^{\frac{s}{3+s}}\,
 (\log(1+\tau))^{q} \, e^{-\frac{\kappa_0}{2} (1+\tau)^{\frac{s}{3+s}} \,
  (\log(1+\tau))^{ -\frac{3+qs}{3}}} \, d\tau \\
&\le C \, e^{-c\frac{\kappa_0}{2} \, (1+t)^{\frac{s}{3+s}} \,
 (\log(1+t))^{ -\frac{3+qs}{3}}},
\ea
$$
for some constants $c,C>0$.
\par
Finally, taking $ -\frac{3+qs}{3} =q$, that is $ q =- \frac{3}{3+s}$, we deduce that there are constants $C,c >0$ such that
$$ x(t) \le C\, e^{- c\, (1+t)^{\frac{s}{3+s}}\,  
 (\log(1+t))^{ - \frac{3}{3+s} } } ,
$$
which completes the proof. 
\end{proof}

{\bf{Acknowledgement}}:  The research leading to this paper was funded by 
the French ``ANR blanche'' project Kibord: ANR-13-BS01-0004.
K.C. is supported by the Fondation Math\'ematique Jacques Hadamard.


\end{document}